\documentclass[11pt
]{a_a_win}

\usepackage[T2A]{fontenc} 
\usepackage[cp1251]{inputenc} 
\usepackage[english,russian]{babel}
\usepackage{amsfonts,mathrsfs,amscd} 
\usepackage{amssymb,latexsym}
 \usepackage{amsmath} 
\usepackage{enumerate} 
\usepackage{verbatim}
\usepackage[mathscr]{eucal} 
\usepackage[dvips]{graphicx} 

\usepackage[hyper]{amsbib}
\usepackage{esint}

\theoremstyle{plain} 
\newtheorem{theorem}{Теорема}

\newtheorem*{thAa}{Теорема A1}
\newtheorem*{thAb}{Теорема A2} 
\newtheorem*{thAc}{Теорема A3}
\newtheorem*{thAd}{Теорема A4}

\newtheorem{lemma}{Лемма}[section]
\newtheorem{propos}{Предложение}[section] 
\newtheorem{corollary}{Следствие}[section] 

\theoremstyle{definition}
\newtheorem{definition}{Определение}[section] 
\newtheorem{remark}{Замечание}[section]

\theoremstyle{plain} 
\newtoks\thehProclaim 
\newtheorem*{Proclaim}{\the\thehProclaim}

\theoremstyle{definition} 
\newtoks{\thehRemark} \newtheorem*{Remark}{\the\thehRemark}

\renewcommand{\leq}{\leqslant} 
\renewcommand{\geq}{\geqslant}

\newcommand{\rad}{\text{\tiny\rm rad}}
\newcommand{\bal}{\rm {bal}}
\newcommand{\Bal}{\rm {Bal}}
\newcommand{\RR}{\mathbb{R}} 
\newcommand{\CC}{\mathbb{C}} 
\newcommand{\NN}{\mathbb{N}} 
\newcommand{\ZZ}{\mathbb{Z}}
 
\newcommand{\e}{\varepsilon}
\newcommand{\const}{{\rm const}}

\DeclareMathOperator{\Har}{har} 
\DeclareMathOperator{\Hol}{Hol} 
\DeclareMathOperator{\Poi}{P}

\DeclareMathOperator{\Zero}{Zero} 
\DeclareMathOperator{\sbh}{sbh} 

\DeclareMathOperator{\dsbh}{\text{$\delta${\rm -sbh}}} 
\DeclareMathOperator{\supp}{supp} 

\DeclareMathOperator{\type}{type} 
\DeclareMathOperator{\ord}{ord}

\DeclareMathOperator{\up}{\text{\rm \tiny up}}
 \DeclareMathOperator{\lw}{\text{\rm \tiny low}}
 
\DeclareMathOperator{\sgn}{sgn}

\DeclareMathOperator{\dd}{\,{\mathrm d\!}}
\renewcommand{\Re}{{\rm Re \,}}
\renewcommand{\Im}{{\rm Im \,}}

\begin{document} 
\title[Выметание конечного рода на систему лучей]{Выметание мер и субгармонических функций на систему лучей. II. Выметания конечного рода и  регулярность роста на одном луче}
	
\author{Б.\,Н.~Хабибуллин, А.\,В.~Шмелева, З.\,Ф.~Абдуллина} 
	
\address{факультет математики и ИТ\\ Башкирский государственный университет\\ 450074, г. Уфа\\ ул. Заки Валиди, 32\\ Башкортостан\\
Россия}
	
\email{Khabib-Bulat@mail.ru}
	
	
\begin{abstract} Расширяются классические выметания мер и субгармонических функций на систему лучей $S$ с общим началом на комплексной плоскости $\CC$. Это позволяет для произвольной субгармонической функции $v$ конечного порядка на $\CC$ строить  $\delta$-субгармоническую на $\CC$ функцию, гармоническую вне $S$, совпадающую с $v$ на $S$ вне полярного множества, того же порядка роста, что и $v$.  Приводятся применения к исследованию   взаимосвязи роста целой функции на $S$ с распределением её  нулей. 
В настоящей второй части работы существенно используются результаты и  подготовительной материал ее первой части.

Библиография: 18 названий. 
\end{abstract}
	
\thanks{Исследование выполнено за счёт гранта Российского
научного фонда (проект № 18-11-00002).
}	
\date{5 сентября 2018 г.}
	
\maketitle

\tableofcontents

\section{Введение}\label{s10}
\subsection{Истоки и предмет исследования}\label{111}

Все обозначения и  определения согласованы с первой частью  \cite{KhI} нашей работы и в значительной мере с исходной для нашего исследования статьей первого из авторов \cite{Kha91}. При необходимости мы их напоминаем. В частности, для подмножества $S$ в {\it комплексной плоскости\/} $\CC$ через $\sbh (S)$, $\Har (S)$, $\dsbh(S)$ и $\Hol (S)$ обозначаем классы соответственно {\it субгармонических, гармонических, $\delta$-субгармонических\/} и {\it голоморфных\/}  функций на каких-либо открытых окрестностях  множества $S$. 

Как уже отмечалось в \cite[введение]{KhI}, в отличие от классического случая выметания на {\it замкнутую систему лучей\/}  $S$ с общим началом  в  $\CC$ выметание конечного рода $q=0,1,2,\dots$ позволит сопоставлять каждой мере $\nu$ конечного порядка из класса борелевских положительных мер $\mathcal M^+(\CC)$ на плоскости с помощью определенного интегрального преобразования некоторый заряд $\nu^{\bal [q]}\in \mathcal M(\CC)$ того же порядка, сосредоточенный на $S$ и называемый выметанием рода $q$ меры $\nu$ на $S$. Наиболее важно при этом то, что для каждой функции $v\in \sbh (\CC)$ с мерой  Рисса $\nu$ можно подобрать, вообще говоря, не единственную, $\delta$-субгармоническую функцию, т.\,е. разность  субгармонических функций \cite{Ar_d}--\cite[3.1]{KhaRoz18},  $v^{{\Bal}}\in \dsbh (\CC)$ с зарядом Рисса $\nu^{\bal [q]}$, для которой  $v^{{\Bal}}=v$ на $S$ вне полярного множества, т.\,е. множества нулевой логарифмической ёмкости \cite{HK}, \cite{Rans}, и  функция $v^{{\Bal}}$ гармоническая в  $\CC\setminus S$. Функцию $v^{{\Bal}}$ называем тоже выметанием функции $v$ на систему лучей $S$. Если функция $v$ конечного порядка роста, то ее выметание $v^{{\Bal}}$ можно подобрать того же порядка роста. При $q=0$ это в точности классическое выметание из первой части \cite{KhI} нашей работы.  Подобная техника позволяет при исследовании поведения субгармонической или целой функции 
на системе лучей $S$ ограничиться исследованием $\delta$-субгармонических функций, гармонических вне $S$, с зарядом Рисса на $S$.

Выметание конечного рода мер и субгармонических функций будет использовано для получения новых критериев вполне регулярного роста субгармонической и/или целой функции конечного порядка на произвольной системе лучей $S$ и, прежде всего, на одном луче в терминах выметания распределения масс и/или нулей на эту систему лучей. Некоторый специальный перенос-смещение мер Рисса субгармонической функции на лучи использовался в совместной статье А.\,Ф.~Гришина и  Т.\,И.~Малютиной  \cite{GrM} для получения, в частности, признаков вполне регулярного роста субгармонической функции $v\in \sbh(\CC)$ на луче в терминах некоторых специальных интегралов от  $v$ по этому лучу \cite[теорема 3]{GrM}, но не в терминах меры Рисса этой функции $v$. Дополнительные результаты и библиографические ссылки по исследованию  полной регулярности роста субгармонической и/или целой функции на лучах можно найти в \cite{KhI}, \cite{Kha91}, \cite{GrM}, \cite{Levin56}.

Используемые в работе интегральные преобразования, представления и ядра в них имеют параллели с представлениями голоморфных функций на полуплоскости из монографии Н.\,В.~Говорова \cite[гл.~1, \S~3]{Gov} по форме и духу, но не по содержанию, хотя  в некотором смысле двойственные глубинные взаимосвязи с представлениями Н.\,В.~Говорова отрицать нельзя.

Ниже в подразделе \ref{singleray} сформулированы теоремы  A1, A2, A4 о выметании на систему лучей $S$ и его применениях лишь в частном случае системы лучей $S$, состоящей из единственного луча:  замкнутой {\it положительной полуоси\/}  
$\RR^+$  на {\it вещественной оси\/} $\RR$. Тем не менее большинство результатов, приведенных и обсуждаемых в  подразделе \ref{singleray},  нам ранее не встречалось и, на наш взгляд, уже представляют интерес. Они достаточно детально иллюстрируют   более общие  конструкции и результаты  настоящей второй части работы из \S\S~\ref{bal_S}--\ref{bqsf}.
Самостоятельный интерес представляет также развитие в подразделе 
\ref{CBsf} известных взаимосвязей между распределением масс/нулей субгармонических/целых функций в угле типа условия Бляшке  и поведением интегралов по лучам в терминах условий типа Ахиезера \cite[гл.~V]{Levin56}, основанное на использовании  вариаций на тему  классической леммы Эдрея и Фукса о малых дугах \cite{EF}, \cite[гл. 1, теорема 7.3]{GO},  леммы о малых интервалах А.\,Ф.~Гришина и М.\,Л.~Содина \cite[лемма 3.2]{GrS}   и их аналогов из   \cite[теорема 8]{GrM}. Эти существенно обобщающие вариации 
леммы о малых интервалах даны в теореме A3  из п.~\ref{inti} и в теореме \ref{thr} из подраздела \ref{Intr} для субгармонических функций произвольного роста.
 
\subsection{Выметание конечного рода  на $\RR^+$}
\label{singleray}


\subsubsection{Выметание конечного рода на $\RR^+$ заряда конечного порядка}\label{b1ray} Здесь мы ограничимся лишь случаем выметания на систему лучей $S$, состоящую из одного луча $\RR^+:=\{x\in \RR\colon x\geq 0 \} \subset \CC$.   
Редукция угла $\angle(0 , 2\pi):=\CC\setminus \RR^+$  к верхней полуплоскости $\CC^{\up}:=\{z\in \CC\colon \Im z>0 \}$ с помощью конформной замены переменных 
\begin{equation}\label{tildez0} 
z\mapsto \sqrt z, \quad z\in \CC\setminus \RR^+,\quad \RR_*^+:=\RR^+\setminus \{0\},
\end{equation} 
где при  извлечении квадратного корня рассматривается аналитическая ветвь, положительная на <<верхнем берегу>> положительной полуоси $\RR^+$, позволяет определить гармонический заряд рода $q$ для   $\CC\setminus \RR^+$. Конкретнее, замкнутый угол $\angle [0, 2\pi]$ рассматриваем как угол между двумя <<различными>> сторонами-лучами  $\{te^{0i}\colon t\in \RR^+\}$ и $\{te^{2\pi i}\colon t\in \RR^+\}$ вместе с этими двумя сторонами, т.е. рассматривается один лист римановой поверхности функции $\sqrt{\;}$. Пусть $z\in \CC\setminus \RR^+$ и $\sqrt B$ --- образ пересечения   $B\cap \angle\,[0, 2\pi]\in \mathcal B(\CC)$ с замкнутым углом $\angle\,[0, 2\pi]$   при редукции \eqref{tildez0} угла
$\angle (0, 2\pi)$ к $\CC^{\up}$. При непустом пересечении $B\cap \RR^+$ образ $\sqrt B$ имеет непустое пересечение как с $\RR^+$, так и с  отрицательной полуосью $-\RR^+$.

{\it Гармонический заряд рода $q\in \NN_0:=\NN\cup \{0\}$ для угла $\angle (0, 2\pi)$} --- функция 
$\Omega_{\CC\setminus \RR^+}^{[q]}\colon \angle(0,2\pi)\times \mathcal B_{\rm b} (\CC) \to \RR$, где $\mathcal B_{\rm b} (\CC)$ 
--- класс ограниченных борелевских подмножеств в $\CC$,
определенная через классическую гармоническую меру $\omega (\cdot,\cdot)$ для верхней полуплоскости $\CC^{\up}$  как 
\begin{equation}\label{se:HC+0}
\Omega_{\CC\setminus \RR^+}^{[q]}(z,B):=
\omega(\sqrt z,\RR\cap \sqrt B\,)+\frac{1}{\pi} 
 \int_{\RR\cap \sqrt B}\Im \frac{t^q-(\sqrt z)^q}{(\sqrt z)^{q}(t-\sqrt z)} \dd t .
\end{equation}
Выметание рода\/ $q$ заряда\/ $\nu$ из угла $\angle(0, 2\pi)$ 
на его границу $\partial \angle(0, 2\pi)$, или 
на $\RR^+$, для $B\subset \mathcal B_b(\CC)$ определяется  равенством 
  \begin{equation}\label{df:bal++}
    \nu_{\RR^+}^{\bal [q]}(B):=\int_{\CC\setminus \RR^+} \Omega_{\CC\setminus \RR^+}^{[q]}(z, B\cap\angle(0,2\pi))\dd \nu(z)+
\nu (B\cap \RR^+ ).
  \end{equation} 
Из  теоремы  \ref{th15}, как отмечено в  замечании \ref{remA1},  следует

\begin{thAa}[{\rm ср. с \cite[Лемма 2.1.2]{KhDD92}}]
\label{th15_0} Пусть $\nu\in \mathcal M(\CC)$ --- заряд конечного типа при порядке $p\in \RR^+$ около $\infty$, т.\,е. $\type_p^{\infty}[\nu]{<}+\infty$ в обозначении из \cite[(2.1t), \S~4]{KhI}, $q:=[2p]\in \NN_0$ --- целая часть числа  $2p$ и для некоторого $r_0>0$ имеем  $D(r_0)\cap \supp \nu=\varnothing$, где $D(r):=\bigl\{z\in \CC\colon |z|<r\bigr\}$.
\begin{enumerate}[{\rm 1.}]
\item\label{c:0nuniii+0} Тогда существуют выметание $\nu_{\RR^+}^{\bal[q]}$ из  $\CC\setminus \RR^+$ на $\RR^+$ и  постоянная\footnote{Как и в \cite[1.3.2]{KhI}, через $\const_{a_1, a_2, \dots}$ обозначаем вещественные постоянные, зависящие от $a_1, a_2, \dots$ и, если не
оговорено противное, только от них.} $C:=\const_{\nu,p}$,
для  которых  
\begin{equation*}
\bigl|\nu_{\RR^+}^{\bal[q]}\bigr|^{\rad}(r)\leq C \left(r^p+r^{q/2}\Biggl|\;\int\limits_{D(r)\setminus \RR^+} \Im \frac{1}{(\sqrt z\,)^q}\dd \nu (z)\Biggr|\right)\quad\text{при всех $r\in \RR^+$},
\end{equation*}
где использованы обозначения $|\nu|$ для  полной вариации заряда $\nu$ и $\nu^{\rad}(r):=\nu\bigl(\overline D(r)\bigr)$, $\overline D(r)$ --- замыкание круга $D(r)$.  

\item\label{c:0nui0} При любом  $p\in \RR^+$ выполнено соотношение  
\begin{equation}\label{es:plog+0}
\bigl|\nu_{\RR^+}^{\bal[2p]}\bigr|^{\rad}(r)=O(r^p\log  r)\quad \text{при $r\to +\infty$}.
\end{equation}

\item\label{c:0nuni+0} Если  $2p\in \RR^+\setminus \NN_0$ --- нецелое число,  то 
выметание $\nu_{ \RR^+}^{\bal[q]}$  конечного типа при том же порядке $p$, т.\,е.  $\type_{p}^{\infty}\bigl[\nu_{ \RR^+}^{\bal[q]}\bigr]<+\infty$.  
 
\item\label{c:0nuiii+0} Если  
$2p\in \NN_0$ и для заряда $\nu$ в угле $\angle(0,2\pi)$  выполнено условие\/ Бляшке рода $p$, а именно:
\begin{equation}\label{ka+0}
\Biggl|\;\int\limits_{D(r)\setminus \RR^+} \Im \frac{1}{z^{p}}\dd \nu (z)\Biggr|=O(1), \quad r\to +\infty,
\end{equation}
 то  $\type_{p}^{\infty}\bigl[\nu_{\RR^+}^{\bal[q]}\bigr]<+\infty$.   
\end{enumerate}
\end{thAa}

\subsubsection{Выметание конечного рода  $\delta$-субгармонической функции на $\RR^+$}\label{b1ray1}

Функции, тождественно равные $-\infty$ и $+\infty$ на $\CC$, обозначаем соответственно через $\boldsymbol{-\infty}\in \dsbh(\CC)$ и $\boldsymbol{+\infty}\in \dsbh(\CC)$; $\dsbh_*(\CC):=\dsbh(\CC)\setminus \{\boldsymbol{\pm \infty}\}$; $\sbh_*(\CC):=\sbh(\CC)\setminus \{\boldsymbol{-\infty}\}$ \cite{Ar_d}--\cite[3.1]{KhaRoz18}.

Пусть $p\in \RR^+$  и  $v\in \dsbh_*(\CC)$. Согласно 
общему определению \ref{df:Bghang} функцию $v^{{\Bal}}\in \dsbh_*(\CC)$ называем {\it выметанием функции\/ $v$ из\/ $\CC \setminus \RR^+$ на $\RR^+$,\/} если 
$v^{{\Bal}}=v$ на $\RR^+$ вне полярного множества  и  сужение $v^{{\Bal}}\bigm|_{\CC\setminus \RR^+}$ --- гармоническая функция на $\CC\setminus \RR^+$. 
Из теоремы \ref{thdsb_a} в подразделе \ref{A2} выводится 

\begin{thAb} Пусть функция $v\in \dsbh_*(\CC)$ c зарядом Рисса $\nu_v$ представима в виде  разности двух субгармонических функций конечного типа при порядке $p\in \RR_*^+$, гармонических в некотором круге $D(r_0)$, $r_0>0$, число $q=[2p]$ --- целая часть числа $2p$. Тогда существует  выметание $v^{{\Bal}}$ на $\RR^+$ с зарядом Рисса $(\nu_v)_{\RR^+}^{\bal[q]}$,  представимое в виде разности  $v^{{\Bal}}=v_+-v_-$ двух   функций  $v_{\pm}\in \sbh_*(\CC)$, гармонических в $\CC\setminus \RR^+$, для которых 
\begin{enumerate}[{\rm (i)}]

\item\label{vvii++R} при  произвольном $p$ имеем
$v_\pm (z)\leq O\bigl(|z|^p\log^2 |z|\bigr)$ при  $z\to \infty$. 

\item\label{vvi+R} при одновременно  нецелых $p$ и $2p$ функции $v_{\pm}$ конечного типа при порядке $p$, т.\,е. $\type_p^{\infty}[v_{\pm}]<+\infty$ в смысле \cite[(2.1t), замечание 2.1]{KhI};

\item\label{vvii+R} если $p\in \NN$, то 
\begin{equation}\label{llogv++}
v_\pm (z)\leq O\bigl(|z|^p\log |z|\bigr) \quad\text{при  $z\to \infty$.}
\end{equation}

\item\label{vvii+++R}  если $p\notin \NN$, но $2p\in \NN$ и одновременно в угле $\angle (0,2\pi)$ выполнено условие\/ Бляшке рода $p$ из \eqref{ka+0},  что в данном случае эквивалентно условию Ахиезера, или, кратко, --- условию\/ {\rm А}, рода $p$ относительно угла $\angle (0,2\pi)$ {\rm (определение \ref{df:clA}):}
\begin{equation}\label{fK:abp+0}
J_{0,2\pi}^{[p]}(r_0,r;v)
\overset{\eqref{fK:abp+}}{:=}2\int_{r_0}^r\frac{v(t)}{t^{p+1}} \dd t=O(1) \quad \text{при $r\to +\infty$},
\end{equation}
то также выполнено соотношение \eqref{llogv++}.
\end{enumerate}
\end{thAb}

\subsubsection{Оценки интеграла от субгармонической функции по интервалам}\label{inti}
В теории роста целых и субгармонических функций на $\CC$ часто могут быть полезными  результаты вспомогательного подраздела \ref{Intr}, о которых упоминалось в конце подраздела \ref{111}.

Для локально ограниченной сверху  функции $v\colon \CC \to \{-\infty \}\cup \RR$
 положим 
\begin{equation}\label{Mvr}
M_v(r):=\sup_{|z|=r}v(z), \quad  M_v^+(r):=\max\bigl\{ 0,M_v(r)\bigr\}, \quad   r\in \RR^+. 
\end{equation}

\begin{thAc} Пусть $v\in \sbh_*(\CC)$ --- функция  c мерой Рисса $\nu_v$, функция $g\colon \RR_*^+\to \RR^+$ монотонная  и 
\begin{equation}\label{r0ab}
r_0\in \RR_*^+, \quad 0<b \leq \frac{1}{4}\,. 
\end{equation}
Если функция $g$ возрастающая, то для некоторой постоянной $C\in \RR^+$ справедливы неравенства
\begin{equation}\label{ginc}
\int\limits_{r_0}^R|v(t)| g(t) \dd t \leq
C\,M_v^+\bigl((1+2b)R\bigr)\,R\,g\bigl((1+4b)R\bigr)
\quad\text{при всех $R\geq r_0$}.
\end{equation}
Если функция $g$ убывающая и строго положительная, то для некоторой постоянной $C'\in \RR^+$ имеют место неравенства
\begin{equation}\label{ubg}
\int\limits_{r_0}^R|v(t)| g(t) \dd t \leq C'M_v^+\bigl((1+2b)R\bigr) \int\limits_{(1-b)r_0}^R g(r) \dd r\text{ при всех $R\geq 2r_0$}. 
\end{equation}
\end{thAc}
Теорема A3 доказана в подразделе \ref{prii} как следствие теоремы \ref{thr}.

\begin{remark} Для функций $v\in \sbh_*(\CC)$ {\it конечного уточненного порядка\/}  часть теоремы A4 
для случая {\it возрастающей\/} функции $g$ может быть получена и из результата А.\,Ф.~Гришина и Т.\,И.~Малютиной \cite[теорема 8]{GrM}.
\end{remark}

\subsubsection{Критерий вполне регулярного роста целой функции на одном луче}\label{crcrg}

Пусть последовательность точек ${\sf Z}=\{{\sf z}_k\}_{k=1,2,\dots}$ из $\CC$ не имеет точек сгущения в $\CC$.
Последовательности точек $\sf Z$ сопоставляем считающую меру $n_{\sf Z}$, равную на каждом подмножестве $S\subset \CC$ числу точек из $\sf Z$, содержащихся в $S$. 
Для интервала $I\subset \RR$ через $\omega\bigl( {\sf z}, I \bigr)\geq 0$ обозначаем делённый на $\pi$ угол, под которым виден интервал $I$ из точки ${\sf z}\in \CC$. Для отрезка $[0,x]\subset \RR^+$ полагаем
\begin{equation}\label{se:HC+0ef}
\Omega_{\CC\setminus \RR^+}^{[q]}\bigl(z,[0,x]\bigr)
\overset{\eqref{se:HC+0}}{:=}
\omega\bigl(\sqrt z,[-\sqrt{x}, \sqrt{x}]\,\bigr)
+\frac{1}{\pi} 
 \int\limits_{-\sqrt{x}}^{\sqrt{x}}\Im \frac{t^q-(\sqrt z)^q}{(\sqrt z)^{q}(t-\sqrt z)} \dd t .
\end{equation}
Заряд $\Bigl(n_{\sf Z}^{\bal [q]}\Bigr)^{\RR^+}$ на $\CC$ с носителем $\supp \Bigl(n_{\sf Z}^{\bal [q]}\Bigr)^{\RR^+}\subset \RR^+$ и определяющей его  функцией распределения 
\begin{equation}\label{nuRA0} 
{\Bigl(n_{\sf Z}^{\bal[q]}\Bigr)}^{\RR^+}(x)\overset{\eqref{df:bal++}}{:=}
\sum\limits_{{\sf z}_k\in \CC\setminus \RR^+}\Omega_{\CC\setminus \RR^+}^{[q]}\bigl( {\sf z}_k,[0,x]\bigr)+n_{\sf Z}\bigl([0,x]\bigr), \quad x\in \RR^+,
\end{equation} 
называем {\it выметанием рода $q$\/} меры $n_{\sf Z}$ или {\it последовательности $\sf Z$ на\/} $\RR^+$.
Функция \eqref{nuRA0} локально ограниченной вариации на $\RR^+$, если $q=[2p]$ и $\sf Z$ --- последовательность  конечной верхней плотности при порядке $p$, т.\,е.
$\type_p^{\infty}[n_{\sf Z}]<+\infty$.
Последовательность нулей целой функции $f$, перенумерованную с учетом кратности, обозначаем через $\Zero_f$. 

\begin{thAd}
Целая функция  $f\neq 0$ с  ${\sf Z}=\Zero_f$, $0\notin \sf Z$, конечного типа при порядке $p\in \RR_*^+$ имеет вполне регулярный рост при том же порядке $p$ на луче $\RR^+$, если и только если в определениях и обозначениях \eqref{se:HC+0ef}--\eqref{nuRA0}  существует предел 
\begin{equation}\label{cvpreg}
\lim_{\substack{x\to +\infty\\ x\notin E}}
x^{[p]+1-p}\fint\limits_0^{+\infty} \frac{1}{x-t}
\, {\Bigl(n_{\sf Z}^{\bal [{\tiny[2p]}]}\Bigr)}^{\RR^+} (t) \,\frac{\dd t}{ \, t^{[p]+1}}, 
\end{equation}
где $[2p]$ --- целая часть числа $2p\in \RR_*^+$, перечеркнутый интеграл $\fint$ определяет главное значению в смысле Коши в точке $x\in \RR^+$, а исключительное множество $E\subset \RR^+$ нулевой относительной линейной меры на $\RR^+$, т.\,е. 
\begin{equation}\label{cvpreg+}
 \lim_{r\to +\infty} \frac{1}{r}\int\limits_{E\cap [0,r]}\dd x=0.
\end{equation}
\end{thAd}
Интеграл в правой части \eqref{cvpreg} --- преобразование Гильберта с точностью до нормирующего множителя  части подынтегрального выражения, не содержащей ядро $1/(x-t)$ и продолженной на $(-\infty,0)$ нулем.     

\section{Гармонический заряд конечного рода}\label{bal_S}
\setcounter{equation}{0}
\subsection{ Гармонический заряд и ядро Пуассона рода $q$}
Напомним, что $\CC^{\up}$ --- открытая верхняя полуплоскость в комплексной плоскости $\CC$.
Через $\mathcal B (\CC)$ обозначаем множество всех борелевских подмножеств в $\CC$, а через $\mathcal B_{\rm b}(\CC)\subset \mathcal B (\CC)$ --- подкласс ограниченных в $\CC$ борелевских подмножеств.

\begin{definition}\label{df:O}
{\it Гармонический заряд рода\/ $q\in \NN_0:=\NN\cup\{0\}$ для\/ $\CC^{\up}$}  ---  это функция
$\Omega_{\CC^{\up}}^{[q]}\colon \CC^{\up}\times \mathcal B_{\rm b} (\CC) \to \RR$,
определенная по правилу 
\begin{subequations}\label{se:HC}
\begin{align} 
\Omega_{\CC^{\up}}^{[q]}(z,B)&
:=\omega_{\CC^{\up}}(z,B)+\frac{1}{\pi}  \int\limits_{B\cap \RR_{\pm\infty}}\Im \frac{t^q-z^q}{z^{q}(t-z)} \dd t 
\tag{\ref{se:HC}a}\label{seHC:b}\\
&=\omega(z,B)+\frac{1}{\pi} 
 \sum_{k=1}^{q} \Im \frac{1}{z^{k}} \,\int_{B\cap \RR_{\pm\infty}} t^{k-1}\dd t,\quad z\in \CC^{\up},
\tag{\ref{se:HC}b}\label{seHC:d}
\end{align}
\end{subequations}
Иногда удобно будет обозначать \eqref{seHC:b} как $\Omega^{[q]}(z,B; \CC^{\up})$, а элемент меры Лебега $\dd t$ на $\RR$ как $\dd \lambda_{\RR}(t)$ или  $\dd \lambda_{\RR}^{[1]}(t)$ (см. и ср. с  \eqref{se:qnub} ниже).
\end{definition}
В сумме из \eqref{seHC:d}, как обычно, при $\ZZ\ni q_1>q_2\in \ZZ$ по определению $\sum_{k=q_1}^{q_2}\dots =0$.
Указание на нижний индекс ${\CC^{\up}}$ в левой части 
\eqref{seHC:b}, как правило, опускаем. Таким образом, классическая гармоническая мера  $\omega$ для верхней полуплоскости   --- это  гармонический заряд  $\Omega^{[0]}$ рода $0$ из \cite[{\bf 3.1}]{KhI}. В отличие от гармонической меры гармонический заряд рода $q\geq 1$ по существу не является положительной мерой. 

Аналогично \cite[(3.5)]{KhI} определим {\it ядро Пуассона рода\/ $q\in \NN_0$ для $\CC^{\up}$} 
	\begin{subequations}\label{df:kP+}
	\begin{align} 
		\Poi^{[q]} (t,z)&:=\frac{1}{\pi} \,\Im \left(\frac{1}{t-z}+\frac{t^q-z^q}{z^{q}(t-z)}\right)
		=\frac{1}{\pi} \, \Im\frac{t^{q}}{ z^{q}(t-z)}
	\tag{\ref{df:kP+}a}\label{df:kPq}
\\ 
=&	\frac{1}{\pi}\, \Im\frac{1}{t-z}+\frac{1}{\pi}\sum_{k=1}^{q}\Im \frac{ t^{k-1}}{z^{k}} \geq 
\frac1{\pi}\sum_{k=1}^{q}\Im \frac{ t^{k-1}}{z^{k}} 
		, \; t\in \RR,\;  z\in 	{\CC}^{\up}.
\tag{\ref{df:kP+}b}\label{{df:kPq}b}
\end{align}
\end{subequations}
Доопределяя  нулем $\Poi^{[q]} (t,z)$ для всех  $z\in \RR\setminus \{0,t\}$, согласно  \cite[(3.6a)]{KhI} и равенству $\omega(x,B)=\delta_x(B)$ для $z=x\in\RR$, при $-\infty<t_1<t_2<+\infty$ имеем 
	\begin{subequations}\label{Omt1t2}
	\begin{align} 
		\Omega^{[q]}\bigl(z, [t_1,t_2]\bigr)&:=\int\limits_{t_1}^{t_2} \Poi^{[q]} (t,z) \dd t\quad\text{при \;$z\in {\CC}^{\up}$}, 
	\tag{\ref{Omt1t2}a}\label{{Omt1t2}a}
\\
\Omega^{[q]}\bigl(z, [t_1,t_2]\bigr)&:=\delta_z\bigl([t_1,t_2]\bigr) \quad\text{при \;$ z\in \RR$}.
\tag{\ref{Omt1t2}b}\label{{Omt1t2}b}
\end{align}
\end{subequations}
\begin{propos}\label{pr:Omes} Пусть $z\in \CC^{\overline \up}\setminus \{0,t\}$, где $\CC^{\overline \up}:=\CC^{\up}\cup \RR$ --- замкнутая верхняя полуплоскость,  $a\in (0,1)$. Тогда 
\begin{subequations}\label{seP:e}
\begin{align} 
 \bigl|\Poi^{[q]} (t,z)\bigr|&\leq \frac{|t|^q}{\pi (1-a)|z|^{q+1}} \quad\text{при  $|t|\leq a|z|$},
\tag{\ref{seP:e}a}\label{{seP:e}c}
\\ 
 \bigl|\Poi^{[q]} (t,z)\bigr|&\leq \frac{|t|^{q-1}}{\pi (1-a)|z|^{q}}
\quad\text{при $a|t|\geq |z|$}.
\tag{\ref{seP:e}b}\label{{seP:e}d}
\end{align}
\end{subequations}
\end{propos}
\begin{proof}  
Из \eqref{df:kPq} при $|t|\leq a|z|$ получаем
\begin{equation}
	\bigl|\Poi^{[q]} (t,z)\bigr|\overset{\eqref{df:kPq}}{\leq}\frac{|t|^q}{\pi |z|^q\bigl(|z|-|t|\bigr)}
	\leq \frac{|t|^q}{\pi |z|^q\bigl(|z|-a|z|\bigr)} \, ,
\end{equation}
что дает \eqref{{seP:e}c}. При $a|t|\geq |z|$ неравенства 
\begin{equation*}
	\bigl|\Poi^{[q]} (t,z)\bigr|\overset{\eqref{df:kPq}}{\leq}\frac{|t|^q}{\pi |z|^q\bigl(|t|-|z|\bigr)}
	\leq \frac{|t|^q}{\pi |z|^q\bigl(|t|-a|t|\bigr)} \,.
\end{equation*}
дают \eqref{{seP:e}d}. 
\end{proof}
\begin{propos}\label{pr:OmesO} Пусть  $-\infty<t_1<t_2<+\infty$, $q\in \NN_0$. Тогда 
для любых $z\in \CC^{\overline \up}\setminus \{0\}$ при $T:=\max\{|t_1|, |t_2|\}$ имеют место неравенства
\begin{subequations}\label{se:Omlu}
\begin{align} 
-\frac1{\pi}(t_2-t_1)\sum_{k=1}^{q}T^{k-1}\Bigl|\Im \frac{1}{z^{k}}\Bigr|&\leq \Omega^{[q]} \bigl(z, [t_1,t_2]\bigr)
\tag{\ref{se:Omlu}a}\label{{se:Omlu}a}
\\ 
 	\leq \omega \bigl(z, [t_1,t_2]\bigr)&+\frac1{\pi}(t_2-t_1)\sum_{k=1}^{q}T^{k-1}\Bigl|\Im \frac{1}{z^{k}}\Bigr|,
\tag{\ref{se:Omlu}b}\label{{se:Omlu}b}
\end{align}
\end{subequations}
а  при $t_2=-t_1=:t\in \RR^+$ справедливы неравенства
\begin{subequations}\label{se:Om+}
\begin{align} 
-\frac2{\pi}\sum_{1\leq m\leq\frac{q+1}{2}}t^{2m-1}\Bigl|\Im \frac{1}{z^{2m-1}}\Bigr|&\leq \Omega^{[q]} \bigl(z, [-t,t]\bigr)
\tag{\ref{se:Om+}a}\label{{se:Om+}a}
\\ 
 \leq \omega \bigl(z, [-t,t]\bigr)+\frac2{\pi}&\sum_{1\leq m\leq\frac{q+1}{2}}t^{2m-1}\Bigl|\Im \frac{1}{z^{2m-1}}\Bigr|.
\tag{\ref{se:Om+}b}\label{{se:Om+}b}
\end{align}
\end{subequations}
\end{propos}
\begin{proof} Из представления \eqref{seHC:d} в силу положительности гармонической меры имеем
\begin{equation}\label{allOm}
0\leq \Omega^{[q]} \bigl(z, [t_1,t_2]\bigr)-\frac1{\pi}\sum_{k=1}^{q}(t_2^k-t_1^k)\Im \frac{1}{kz^{k}}\leq \omega \bigl(z, [t_1,t_2]\bigr)\leq 1, \quad z\in \CC^{\up},
\end{equation} 
откуда
\begin{subequations}
\begin{align*} 
-\frac1{\pi}(t_2-t_1)\sum_{k=1}^{q}\Bigl|\sum_{l=0}^{k-1}t_2^{k-1-l}t_1^l\Bigl|\,\Bigl|\Im \frac{1}{kz^{k}}\Bigr|&\leq \Omega^{[q]} \bigl(z, [t_1,t_2]\bigr)
\\ 
 	\leq \omega \bigl(z, [t_1,t_2]\bigr)&+\frac1{\pi}\sum_{k=1}^{q}\Bigl|\sum_{l=0}^{k-1}t_2^{k-1-l}t_1^l\Bigl|\,\Bigl|\Im \frac{1}{kz^{k}}\Bigr|,
\end{align*}
\end{subequations}
что влечет за собой \eqref{se:Omlu}. Аналогично из \eqref{allOm} имеем
\begin{equation*}
0\leq \Omega^{[q]} \bigl(z, [-t,t]\bigr)-\frac2{\pi}\sum_{1\leq m\leq\frac{q+1}{2}}\Im \frac{t^{2m-1}}{(2m-1)z^{2m-1}}\leq \omega \bigl(z, [-t,t]\bigr)\leq 1,
\end{equation*}
что дает \eqref{se:Om+}.
\end{proof}

\subsection{Оценки гармонического заряда рода $q$}
\begin{propos}\label{pr:OmesOa}
Пусть $-\infty<t_1<t_2<+\infty$, $q\in \NN_0$,  $a\in (0,1)$, $z\in {\CC}^{\overline \up}$. Если 
\begin{equation}\label{TT}
T:=\max \{|t_1|,|t_2|\}\leq a|z|
\end{equation}
то 
\begin{equation}\label{est:O1}
	\bigl|\Omega^{[q]} \bigl(z, [t_1,t_2]\bigr)\bigr|\leq 
	(t_2-t_1)\, \frac{2T^q}{\pi (1-a)|z|^{q+1}}\,.
	\end{equation}
Если выполнено условие
\begin{equation}\label{tt}
	\min\limits_{t\in [t_1, t_2]} |t|\geq |z|/a,
\end{equation}
 то при $q>0$ имеет место неравенство
\begin{equation}\label{est:O2}
	\bigl|\Omega^{[q]} \bigl(z, [t_1,t_2]\bigr)\bigr|\leq 
	(t_2-t_1) \, \frac{T^{q-1}}{\pi (1-a)|z|^{q}}\, .
\end{equation}
Кроме того,  если выполнено одно из условий \eqref{TT} или \eqref{tt},  а также если  одновременно выполнены условия $t_1\geq 0$ и \begin{equation}\label{minMaa} -1< \cos \arg z\leq
\frac{2a\sqrt{t_1t_2}}{t_1+t_2} \,, \end{equation} то  
\begin{equation}\label{Imq}
	\Omega^{[q]} \bigl(z, [t_1,t_2]\bigr)\leq (t_2-t_1)\,\frac{1}{\pi}
	\left(\frac{2}{(1-a)^2}\Im \frac{1}{\bar z}	+\sum_{k=2}^q\,T^{k-1}\Bigl|\Im \frac{1}{z^k}\Bigr|\right)
\end{equation}
\end{propos}
\begin{proof}
Согласно   \eqref{Omt1t2} из \eqref{{seP:e}c}  следует ($\sgn=$сигнум)
\begin{multline*}
	\bigl|\Omega^{[q]} \bigl(z, [t_1,t_2]\bigr)\bigr|\leq \frac{1}{\pi (1-a)|z|^{q+1}}\int_{t_1}^{t_2}|t|^{q}\dd t
=\frac{\bigl||t_2|^{q+1}-\sgn(t_1t_2) |t_1|^{q+1}\bigr|}{\pi (1-a)(q+1)|z|^{q+1}}
\\
\leq \frac{\bigl||t_2|-\sgn(t_1t_2) |t_1|\bigr|}{\pi (1-a)(q+1)|z|^{q+1}}
\max\left\{\sum_{k=0}^q|t_2|^{q-k}|t_1|^{k}, |t_1|^{q}+|t_2|^q\right\}
\\
\leq \frac{\bigl||t_2|-\sgn(t_1t_2) |t_1|\bigr|\max\{q+1,2\}T^q}{\pi (1-a)(q+1)|z|^{q+1}}
\leq \frac{2(t_2-t_1)T^q}{\pi (1-a)|z|^{q+1}}\,,
\end{multline*}
что и дает \eqref{est:O1}. Из условия  $\min\limits_{t\in [t_1, t_2]} |t|\geq |z|/a$, в частности, следует, что пара точек
 $t_1, t_2$ расположена на $\RR$ по одну сторону от нуля. С учетом этого и согласно \eqref{Omt1t2} из \eqref{{seP:e}d} при $q>0$ получаем
\begin{multline*}
	\bigl|\Omega^{[q]} \bigl(z, [t_1,t_2]\bigr)\bigr|\leq \frac{1}{\pi (1-a)|z|^{q}}\int_{t_1}^{t_2}|t|^{q-1}\dd t
=\frac{\bigl||t_2|^{q}-|t_1|^{q}\bigr|}{\pi (1-a)q|z|^{q}}
\\
\leq \frac{\bigl||t_2|-|t_1|\bigr|}{\pi (1-a)q|z|^{q}}\sum_{k=0}^{q-1}|t_2|^{q-1-k}|t_1|^{k}\leq
\bigl||t_2|-|t_1|\bigr|\frac{qT^{q-1}}{\pi (1-a)q|z|^{q}}\,,
\end{multline*}
что и дает \eqref{est:O2}. 
Для вывода последней  оценки сверху  \eqref{Imq} достаточно воспользоваться верхней оценкой  \eqref{{se:Omlu}b}
и оценками гармонической меры  из \cite[предложение 3.3]{KhI}, в которых условия \cite[(3.20)]{KhI} и \cite[(3.22)]{KhI}, дающие оценки \cite[(3.21)]{KhI} и \cite[(3.23)]{KhI}, совпадают с \eqref{tt} и \eqref{TT} соответственно. При условии \eqref{minMaa} вместе с $t_1\geq 0$ в соответствии с оценками  \cite[предложение 3.5, (3.30)--(3.31)]{KhI} получаем \eqref{Imq}. При этом следует учесть очевидное равенство $|\Im (1/z)|=\Im (1/\bar z)$ при $0\neq z\in \CC^{\overline \up}$.
\end{proof}

\section{Выметание конечного рода  заряда}\label{ss:bmu}
\setcounter{equation}{0}

\subsection{Выметание рода $q$ из верхней полуплоскости и его оценки}\label{sss:mu} 

\begin{definition}\label{df:Bal}  Пусть $q\in \NN_0$, $\nu \in \mathcal M(\CC)$. Заряд $\nu^{\bal [q]}\in \mathcal M(\CC)$, определяемый  на каждом  $B\in \mathcal B_{\rm b} (\CC)$ равенством 
  \begin{equation}\label{df:bal}
    \nu_{\CC_{\overline{\lw}}}^{\bal [q]}(B):=:\nu^{\bal [q]}(B):=\int_{\CC^{\up}} \Omega^{[q]}(z, B\cap \CC^{\up})\dd \nu(z)+\nu\bigl(B\cap {{\CC}_{\overline \lw}}\bigr),
  \end{equation} 
  называем {\it выметанием рода\/ $q$ заряда\/ $\nu$}
  из верхней полуплоскости $\CC^{\up}$ на замкнутую нижнюю полуплоскость $\CC_{\overline{\lw}}:=\CC \setminus \CC^{\up}$; $\CC_{\lw}:=-\CC^{\up}$ 
 \end{definition}
При $q=0$, очевидно,  $\nu^{\bal}=\nu^{\bal [0]}$ \cite[определение 4.2]{KhI},
а выметание рода $q$ не меняет сужение $\nu\bigm|_{{\CC}_{\lw}}$ на нжнюю полуплоскость $\CC_{\lw}$, а его носитель $\supp \nu^{\bal [q]}$ содержится в $\CC_{\overline{\lw}}$. 
Если  выметания $(\nu^{\pm})^{\bal [q]}$ рода $q$ существуют для верхней и нижней вариаций $\nu^{\pm}$, то, очевидно, 
\begin{equation}\label{vnvbal}
\nu^{\bal [q]}=(\nu^+)^{\bal [q]}-(\nu^-)^{\bal [q]}.
\end{equation}

Для заряда  $\nu\in \mathcal M(\CC)$ функцию распределения сужения $\nu\bigm|_{\RR}$ на $\RR$ обозначаем, как и в \cite[(1.9)]{KhI} (см. также 
\eqref{nuRA0}
), через 
\begin{equation}\label{nuR} 
\nu^{\RR}(x):=\begin{cases} -\nu\bigl( [x,0) \bigr) \; &\text{ при } x<0 ,\\ \nu\bigl([0,x] \bigr) \; &\text{ при } x\geq 0, \end{cases} 
\quad |\nu|^{\RR}=(\nu^+)^{\RR}+(\nu^-)^{\RR}.
\end{equation}

\begin{propos}\label{pr:b}  Пусть $\nu\in \mathcal M(\CC)$ --- заряд с компактным носителем, $q\in \NN_0$ и для некоторого числа $r_0>0$ сходятся (конечны) интегралы
\begin{equation}\label{i0r0}
\int_{D(r_0)\cap \CC^{\up}} \Im \frac{1}{z^k}\dd \nu (z),\quad\text{при  $1\leq k\leq q$}.
\end{equation}
Тогда в обозначении \eqref{nuR} для функции распределения  $\bigl(\nu^{\bal [q]}\bigr)^{\RR}$ 	выметания $\nu^{\bal [q]}$  рода $  q$ на $\RR$   заряда $\nu$   по всем $t\in \RR$ имеет место тождество
\begin{equation}\label{dqt}
\bigl(\nu^{\bal [q]}\bigr)^{\RR}(t)\equiv (\nu^{\bal})^{\RR}(t)
+\frac{1}{\pi}\sum_{k=1}^q\left(\int_{\CC^{\up}} \Im \frac{1}{z^k}\dd \nu (z)\right)\, \frac{1}{k}\,t^{k}\,,
\end{equation}
 для верхней и полной вариаций $\bigl(\nu^{\bal [q]}\bigr)^+$ и  $\bigl|\nu^{\bal [q]}\bigr|$ имеем неравенства
\begin{subequations}\label{se:qnu}
\begin{align} 
\bigl(\nu^{\bal [q]}\bigr)^+ &\leq |\nu^{\bal [q]}|\leq |\nu^{\bal}|
+\frac{1}{\pi}\sum_{k=1}^q\left|\int_{\CC^{\up}} \Im \frac{1}{z^k}\dd \nu (z)\right| \lambda_{\RR}^{[k]},
\tag{\ref{se:qnu}a}\label{se:qnua}
\\
\intertext{где мера $\lambda_{\RR}^{[k]}$ определена через ее плотность равенством}
d\lambda_{\RR}^{[k]}(t)&:=t^{k-1}\dd \lambda_{\RR}(t)=t^{k-1}\dd t ,\quad t\in \RR,
\tag{\ref{se:qnu}b}\label{se:qnub}
\\
\intertext{а для нижней вариации в случае положительной меры $\nu$ имеем}
\bigl(\nu^{\bal [q]}\bigr)^-&\leq \frac{1}{\pi}\sum_{k=1}^q\left|\int_{\CC^{\up}} \Im \frac{1}{z^k}\dd \nu (z)\right| \lambda_{\RR}^{[k]} \quad \text{при  $\nu\in \mathcal M^+(\CC)$.}
\tag{\ref{se:qnu}c}\label{se:qnuc}
\end{align}
\end{subequations}
\end{propos}
Комментарий к  сходимости интегралов \eqref{i0r0} в терминологии \cite[2.2]{KhI}
--- 
\begin{lemma}\label{l00}  Если заряд $\nu$ принадлежит классу сходимости при порядке исчезания $q$ около нуля, то интегралы 
\eqref{i0r0} сходятся.
\end{lemma}
\begin{proof}[Доказательство леммы]
По условию леммы и по  \cite[предложение 2.2(ii)]{KhI} с  функцией $f:=|\nu|^{\rad}$ сходится интеграл вида \cite[(2.11)]{KhI}:
\begin{equation}\label{0cv0}
	\int_{0}^{r_0} \frac{\dd |\nu|^{\rad}(t)}{t^{q}}=
	\int_{D(r_0)} \frac{1}{|z|^k}\dd |\nu| (z) \geq \int_{D(r_0)} \Bigl|\Im \frac{1}{z^k}\Bigr|\dd |\nu| (z),
\end{equation}
откуда следует сходимость интегралов \eqref{i0r0} при всех $1\leq k\leq q$.  
\end{proof}
\begin{proof}[Доказательство предложения \ref{pr:b}] 
Для доказательства \eqref{dqt} рассмотрим заряды на $\RR$ как вещественные меры Радона. Пусть 
 $f\colon \RR\to \RR$ --- произвольная финитная непрерывная функция. 
Из определения \ref{df:Bal}  и явного  вида 
 гармонического заряда 
рода $q$ в определении \ref{df:O} в виде \eqref{seHC:d} 
\begin{multline*} 
\int_{\RR} f(t)\dd\, \bigl(\nu^{\bal [q]}\bigr)^{\RR}(t)= \int_{\RR} f(t)\dd \,(\nu^{\bal})^{\RR}(t) \\
+\frac{1}{\pi}\sum_{k=1}^q\int_{\RR}f(t) \left(\, \int_{\CC^{\up}} \Im \frac{t^{k-1}}{z^k}\dd \nu (z)\right)\dd \,\lambda_{\RR}(t)
\\
=\int_{\RR} f(t)\dd\, (\nu^{\bal})^{\RR}(t) +\frac{1}{\pi}\int_{\RR}f(t) \left(\int_{\CC^{\up}} 
\sum_{k=1}^q \Im \frac{1}{z^k}\dd \nu (z)\right)\dd \frac{t^k}{k}
,\quad t\in \RR, 
\end{multline*}
что и означает равенство \eqref{dqt}. По определениям верхней и полной вариаций заряда тождество  \eqref{dqt} дает \eqref{se:qnu}.  Оценка \eqref{se:qnuc} также следует из \eqref{dqt}, где для положительной меры $\nu$
мера $\nu^{\bal}$  положительна и в формировании нижней вариации $\bigl(\nu^{\bal [q]}\bigr)^-$ 
не участвует.
\end{proof}

\begin{propos}\label{pr:infty}
Пусть $q\in \NN_0$,  $a\in (0,1)\subset \RR^+$ и 
\begin{subequations}\label{T}
\begin{align}
-\infty<t_1<t_2<+\infty,& \quad  T:=\max\bigl\{|t_1|,|t_2|\bigr\},
\tag{\ref{T}t}\label{tT}\\
\nu\in \mathcal M(\CC),&\quad	\supp \nu \subset \CC^{\overline{\up}}\setminus D(T/a). 
\tag{\ref{T}n}\label{suppnua}
\end{align}
\end{subequations}
Если сужение $\nu\bigm|_{\CC^{\up}}$ заряда $\nu$ на верхнюю полуплоскость   принадлежит классу сходимости при порядке роста $q+1$ (около бесконечности), то 
равенством \eqref{df:bal} корректно определено выметание рода $q$ заряда $\nu$ из определения\/ {\rm \ref{df:Bal}} 
и 	для функции распределения $\bigl|\nu^{\bal [q]}\bigr|^{\RR}$, определенной в \eqref{nuR}, полной вариации $\bigl|\nu^{\bal [q]}\bigr|$ имеет место оценка
\begin{equation}\label{est_nuBr}
\bigl|\nu^{\bal [q]}\bigr|^{\RR}(t_2)-\bigl|\nu^{\bal [q]}\bigr|^{\RR}(t_1)\leq 
\frac{2(q+1)(t_2-t_1)T^q}{\pi (1-a)} \int\limits_{T/a}^{+\infty}\frac{|\nu|^{\rad}(t)}{t^{q+2}}\dd t, 
\end{equation}
где интеграл справа сходится (конечен).
\end{propos}
\begin{proof} По определению \ref{df:Bal}  в обозначениях   $\bigl|\nu^{\bal [q]}\bigr|$  и $|\nu|$ для полных  вариаций зарядов 
 соответственно $\nu^{\bal [q]}$  и $\nu$ имеем оценки
\begin{equation*}
\hspace{-3mm}	\bigl|\nu^{\bal [q]}\bigr|^{\RR}(t_2)-\bigl|\nu^{\bal [q]}\bigr|^{\RR}(t_1)
=  \bigl|\nu^{\bal [q]}\bigr|\bigl([t_1,t_2]\bigr)\overset{\eqref{df:bal}}{\leq} 
	\int\limits_{T/a\leq |z|} 	\bigl|\Omega^{[q]}\bigl(z, [t_1,t_2]\bigr)\bigr|\dd |\nu|(z).
\end{equation*}
Отсюда, используя  оценку \eqref{est:O1} при условии \eqref{TT},  основанном на \eqref{suppnua}, 
\begin{multline*}
\bigl(\nu^{\bal [q]}\bigr)^{\RR}(t_2)-\bigl(\nu^{\bal [q]}\bigr)^{\RR}(t_1)\leq 
\frac{2(t_2-t_1) T^{q}}{\pi (1-a)} \int\limits_{T/a\leq |z|}\frac{1}{|z|^{q+1}}\dd |\nu|(z)
\\
=\frac{2(t_2-t_1) T^{q}}{\pi (1-a)} \int_{T/a}^{+\infty} 	\frac{\dd |\nu|^{\rad}(t)}{t^{q+1}}\,.
\end{multline*}
На основе \cite[предложение 2.1(ii)]{KhI} из принадлежности заряда $\nu$ классу сходимости при порядке роста 
$q+1$ около бесконечности следует, что  $\type_{q+1}^{\infty}[\nu]=0$ и возможно интегрирование по частям вида \cite[(2.5)]{KhI} с $p:=q+1$ и $f:=|\nu|^{\rad}$ для последнего интеграла. Это дает  продолжение оценки правой части   
  в виде
\begin{equation*} 
\leq \frac{2(t_2-t_1)T^q}{\pi (1-a)} 
\left(-\frac{a^{q+1}}{T^{q+1}} |\nu|^{\rad}
\left(\frac{T}{a}\right)+(q+1)\int_{T/a}^{+\infty}\frac{|\nu|^{\rad}(t)}{t^{q+2}}\dd t\right) <+\infty,
\end{equation*}
откуда получаем оценку \eqref{est_nuBr}.
\end{proof}
Достаточные  условия  существования выметания конечного рода  дает

\begin{theorem}\label{th:qbalo}
Пусть $q\in \NN_0$. Если заряд $\nu\in \mathcal M(\CC)$ удовлетворяет условию \eqref{i0r0},  а сужение $\nu\bigm|_{\CC^{\up}}$ заряда $\nu$ на верхнюю полуплоскость принадлежит классу сходимости при порядке роста $q+1$ около бесконечности\/ {\rm \cite[2.1]{KhI}}, то существует выметание $\nu^{\bal [q]}\in \mathcal M(\CC)$ рода $q$. Если заряд
 $\nu \in \mathcal M (\CC)$ сосредоточен вне нуля, т.\,е.  
\begin{equation}\label{Dr0}
D(r_0)\cap \supp \nu =\varnothing \quad\text{для некоторого числа $r_0>0$},
\end{equation} 
то 
\begin{equation}\label{Dr0q}
|\nu^{\bal[q]}|^{\rad} (t)=O(t^{q+1})\quad \text{при $t\to 0$}.
\end{equation}
\end{theorem}
\begin{proof} Не умаляя общности, можно считать, что заряд $\nu$ сосредоточен в $\CC^{\overline \up}$.  Для какого-либо  числа $r_0>0$ положим $\nu_0:=\nu\bigm|_{D(r_0)}$ --- сужение заряда $\nu$ на круг $D(r_0)$, $\nu_{\infty}:=\nu-\nu_0$. По предложениям  \ref{pr:b} и \ref{pr:infty} существуют соответственно $(\nu_0)^{\bal[q]}$ и $(\nu_{\infty})^{\bal[q]}$, а значит и $\nu^{\bal[q]}=(\nu_0)^{\bal[q]}+(\nu_{\infty})^{\bal[q]}$. 
Положим $a=1/2$. Применяя дважды   \eqref{est_nuBr} с $t_1=0< t_2=t<r_0/2$ и 
$-r_0/2<t=t_1<t_2=0$ соответственно,  получаем 
\begin{equation*}
\bigl|(\nu_{\infty})^{\bal[q]}\bigr|^{\rad} (t)\leq 
 \frac{4(q+1)t^{q+1}}{\pi (1-1/2)} \int\limits_{r_0}^{+\infty}\frac{|\nu_{\infty}|^{\rad}(t)}{t^{q+2}}\dd s.
\end{equation*}
Поскольку при условии \eqref{Dr0} $\nu=\nu_{\infty}$, последняя оценка дает \eqref{Dr0q}.
\end{proof}
\subsection{Выметание конечного рода  заряда конечного порядка}\label{ss:bmuf}
Обобщение  \cite[теорема 1]{KhI} для выметания конечного рода  $q\in \NN$  ---
\begin{theorem}\label{th:cupq} Пусть $p\in \RR^+$,  $r_0\in \RR^+_*$,  заряд  $\nu\in \mathcal M(\CC)$ конечного типа при  порядке $p$ удовлетворяет условию \eqref{Dr0}, $q:=[p]$ --- целая часть числа $p$. Тогда определено выметание $\nu^{\bal[q]}\in \mathcal M(\CC)$	из $\CC^{\up}$. При этом для  функции  распределения $\bigl|\nu^{\bal[q]}\bigr|^{\RR}$ на $\RR$, определенной в \eqref{nuR}, при любых  
\begin{equation}\label{0rt0q}
 t\in \RR, \quad 0\leq r\leq |t|, \quad a\in (0,1)\subset \RR^{+}
\end{equation} 
 для $\nu^{\overline \up}:=\nu\bigm|_{\CC^{\overline \up}}$ имеет место  оценка
\begin{multline}\label{nu:estTosq}
\bigl|\nu^{\bal[q]}\bigr|^{\RR}(t+r)-\bigl|\nu^{\bal[q]}\bigr|^{\RR}(t-r)\leq
\const_{\nu,p,a}rt^{p-1}+ |\nu^{\overline{\up}}|(t,r) 
\\+r\biggl(\int_r^{at}\frac{|\nu^{\overline{\up}}|(t,s)}{s^2}\dd s\biggl)^+
+\const_prt^{q-1}\left|\int_{D(t)\setminus D(r_0)} \Im \frac{1}{z^q}\dd \nu^{\overline \up} (z)\right|,
\end{multline}
где для числа $x\in \RR$ использовано обозначение $x^+:=\max\{0,x \}\in \RR^+$.
\end{theorem}
\begin{proof} Заряд $\nu$ при  $\type_p^{\infty}[\nu]<+\infty$ принадлежит классу сходимости при порядке роста $q+1$ около $\infty$ \cite[2.1, (2.2)]{KhI}. Вместе с условием  \eqref{Dr0} это по теореме \ref{th:qbalo} обеспечивает существование выметания $\nu^{\bal[q]}$. Для определенности в \eqref{0rt0q} можем считать, что $t>0$. Достаточно рассмотреть случай $q\in \NN$. Положим 
\begin{equation}\label{df:Tnuo}
T:=t+r, \quad \nu_0:=\nu\bigm|_{D(T/a)}, \quad \nu_{\infty} :=\nu-\nu_0=\nu\bigm|_{\CC\setminus D(T/a)} .
\end{equation}
Для заряда $\nu_0$ по неравенству \eqref{se:qnua} с \eqref{se:qnub} из предложения \ref{pr:b} имеем   
\begin{multline}\label{est:nu0}
\bigl|(\nu_0)^{\bal [q]}\bigr|^{\RR}(t+r)-\bigl|(\nu_0)^{\bal [q]}\bigr|^{\RR}(t-r)
\leq \bigl|(\nu_0)^{\bal}\bigr|^{\RR}(t+r)-\bigl|(\nu_0)^{\bal}\bigr|^{\RR}(t-r)\\
+\frac{1}{\pi}\sum_{k=1}^q
\int_{t-r}^{t+r}\left|\int_{D(T/a)\setminus D(r_0)} \Im \frac{1}{z^k}\dd \nu^{\overline \up} (z)\right| s^{k-1} \dd s \\
\leq \bigl|(\nu_0)^{\bal}\bigr|^{\RR}(t+r)-\bigl|(\nu_0)^{\bal}\bigr|^{\RR}(t-r)
+ \sum_{k=1}^{q-1} \frac{1}{\pi k}\bigl((t+r)^k-(t-r)^k\bigr)\int_{r_0}^{T/a} \frac{1}{s^{k}} \dd \nu^{\rad}(s)
\\+\frac{1}{\pi}
\bigl((t+r)^q-(t-r)^q\bigr) \left|\int_{D(T/a)\setminus D(r_0)} \Im \frac{1}{z^q}\dd \nu^{\overline \up} (z)\right| \\
\leq \Bigl( \bigl|(\nu_0)^{\bal}\bigr|^{\RR}(t+r)-\bigl|(\nu_0)^{\bal}\bigr|^{\RR}(t-r) \Bigr)\\
+\const_{q,\nu,a} rt^{p-1}+\const_q rt^{q-1}\left|\int_{D(T/a)\setminus D(r_0)} \Im \frac{1}{z^q}\dd \nu^{\overline \up} (z)\right|,
\end{multline}
где при последнем переходе использовано интегрирование по частям для интегралов под знаком $\sum_{k=1}^{q-1}\cdots$ и оценка $|\nu|^{\rad}(s)\leq \const_{\nu}s^p$, справедливая ввиду \eqref{Dr0}
для всех $s\in \RR^+$. При $q=0$ эта оценка также верна без последних двух слагаемых в правой части \eqref{est:nu0}. Для оценки первой скобки из правой части \eqref{est:nu0} воспользуемся теоремой 
\cite[теорема 1]{KhI}
\begin{equation}
t_1:=t-r ,\quad t_2:=t+r,  \quad t_0:=t,  \quad T\overset{\eqref{df:Tnuo}}{:=}t+r, \quad (\nu_0)^{\overline \up}:=
\nu_0\bigm|_{\CC^{\overline \up}} 
\end{equation}
в обозначениях \cite[(4.3)]{KhI}  с учетом \cite[(4.4)]{KhI}:
\begin{multline}\label{es:mnu0}
\bigl|(\nu_0)^{\bal}\bigr|^{\RR}(t+r)-\bigl|(\nu_0)^{\bal}\bigr|^{\RR}(t-r)  
{\leq} |\nu^{\overline{\up}}|(t,r) 
+ \const_{\nu, a}rT^{p-1}\\
+\frac{r}{(1-a)^2}\int\limits_{D(T/a)\setminus D(T)}\Bigl|\Im \frac{1}{z}\Bigr|\dd |(\nu_0)^{\overline{\up}}|(z)
+ r\int_r^{aT}\frac{|\nu^{\overline{\up}}|(t,s)}{s^2}\dd s.
\end{multline}
Предпоследний  интеграл в  \eqref{es:mnu0} оцениваем интегрированием по частям:
\begin{equation}\label{es:mnu0+}
\int\limits_{D(T/a)\setminus D(T)}\Bigl|\Im \frac{1}{z}\Bigr|\dd |(\nu_0)^{\overline{\up}}|(z)\leq 
\int_{T}^{T/a} \frac{1}{s} \dd \nu^{\rad}(s)\leq \const_{\nu,a}T^{p-1}.
\end{equation}
Таким же путем оцениваем последний интеграл в \eqref{es:mnu0}, равный
\begin{multline}\label{es:mnu0++}
\biggl(\int_r^{at}+\int_{at}^{aT}\biggr)\frac{|\nu^{\overline{\up}}|(t,s)}{s^2}\dd s\leq
\int_r^{at}\frac{|\nu^{\overline{\up}}|(t,s)}{s^2}\dd s
\\+|\nu|^{\rad}(t+aT)\int_{at}^{aT}\frac{1}{s^2}\dd s\leq 
\biggl(\int_r^{at}\frac{|\nu^{\overline{\up}}|(t,s)}{s^2}\dd s\biggl)^+
+\const_{\nu,a}T^p\frac{1}{t}.
\end{multline} 
Из оценок \eqref{es:mnu0}--\eqref{es:mnu0++} ввиду неравенств $t\leq T\leq 2t$ получаем 
\begin{multline}\label{est:nu0a}
 \bigl|(\nu_0)^{\bal}\bigr|^{\RR}(t+r)-\bigl|(\nu_0)^{\bal}\bigr|^{\RR}(t-r) 
\\{\leq} |\nu^{\overline{\up}}|(t,r) 
+ \const_{\nu, a}rt^{p-1}
+ r\biggl(\int_r^{at}\frac{|\nu^{\overline{\up}}|(t,s)}{s^2}\dd s\biggl)^+.
\end{multline}
Модуль последнего интеграла в \eqref{est:nu0} при $q\in \NN$ равен 
\begin{multline}\label{est:nu0b}
\left|\biggl(\int_{D(t)\setminus D(r_0)} +\int_{D(T/a)\setminus D(t)}\biggr)\Im \frac{1}{z^q}\dd \nu^{\overline \up} (z)\right|
\leq \left|\int_{D(t)\setminus D(r_0)} \Im \frac{1}{z^q}\dd \nu^{\overline \up} (z)\right|
\\+\int_{t}^{T/a}\frac{1}{s^q}\dd \nu^{\rad} (s)
\leq \left|\int_{D(t)\setminus D(r_0)} \Im \frac{1}{z^q}\dd \nu^{\overline \up} (z)\right|+\const_{\nu,a,p}t^{p-q}.
\end{multline}
Таким образом, из  \eqref{est:nu0}, \eqref{est:nu0a} и \eqref{est:nu0b} получаем 
оценку 
\begin{multline}\label{fornu0}
\bigl|(\nu_0)^{\bal [q]}\bigr|^{\RR}(t+r)-\bigl|(\nu_0)^{\bal [q]}\bigr|^{\RR}(t-r)
\leq \const_{\nu,a,p}rt^{p-1}+ |\nu^{\overline{\up}}|(t,r) 
\\+r\biggl(\int_r^{at}\frac{|\nu^{\overline{\up}}|(t,s)}{s^2}\dd s\biggl)^+
+\const_qrt^{q-1}\left|\int_{D(t)\setminus D(r_0)} \Im \frac{1}{z^q}\dd \nu^{\overline \up} (z)\right|.
\end{multline}

Для заряда $\nu_{\infty}$ из \eqref{df:Tnuo} по неравенству \eqref{est_nuBr} предложения \eqref {pr:infty} 
\begin{multline*}
\bigl|(\nu_{\infty})^{\bal [q]}\bigr|^{\RR}(t_2)-\bigl|(\nu_{\infty})^{\bal [q]}\bigr|^{\RR}(t_1)\leq 
 \frac{2(q+1)(t_2-t_1)T^q}{\pi (1-a)} \int\limits_{T/a}^{+\infty}\frac{|\nu|^{\rad}(s)}{s^{q+2}}\dd s\\
\leq \const_{p,a}\cdot rT^q \cdot \const_{\nu} \int_{T/a}^{+\infty} \frac{s^{p}}{s^{q+2}} \dd s
\leq \const_{\nu,p,a} rT^{p-1}\leq \const_{\nu,p,a}rt^{p-1}
\end{multline*}
для любого $t\in \RR^+$. Отсюда и из \eqref{df:Tnuo} и \eqref{fornu0} получаем нужное \eqref{nu:estTosq}.
\end{proof}

 
\begin{corollary} Пусть выполнены условия теоремы\/ {\rm \ref{th:cupq}} и в дополнение 
\begin{enumerate}[{\rm (i)}]
\item\label{aei} при некотором $\e>0$ заряд $\nu^{\overline\up}$ сосредоточен вне пары замкнутых углов $\angle[0,\e]\cup \angle [\pi-\e, \pi]$,
\item\label{aeii} для заряда $\nu^{\overline\up}$ выполнено условие Бляшке
\begin{equation}\label{cBq}
\left|\int_{D(t)\setminus D(r_0)} \Im \frac{1}{z^q}\dd \nu^{\overline \up} (z)\right|=O(1)\quad\text{при $t\to +\infty$}.
\end{equation}
\end{enumerate}
Тогда найдутся постоянные $C\in \RR^+$ и $t_0\in \RR_*^+$, с которыми 
\begin{equation}\label{trnua}
\bigl|\nu^{\bal[q]}\bigr|^{\RR}(t+r)-\bigl|\nu^{\bal[q]}\bigr|^{\RR}(t-r)\leq Crt^{p-1}\; \text{при всех $|t|\geq t_0$, $r\in [0,1]$}.
\end{equation}
\begin{proof} Выберем $a=\min \{\arctg \e, 1/2\}$. Тогда по условию \eqref{aei} для некоторого  $t_0\in \RR_*^+$ при всех $|t|\geq t_0$ полукруги $\overline D\bigl(t,a|t|\bigr)\cap \CC^{\overline{\up}}$ не пересекаются с носителем $\supp \nu^{\overline{\up}}$. Следовательно, в правой части 
\eqref{nu:estTosq} второе и третье слагаемое равны нулю и
\begin{multline*}\label{nu:estTe}
\bigl|\nu^{\bal[q]}\bigr|^{\RR}(t+r)-\bigl|\nu^{\bal[q]}\bigr|^{\RR}(t-r)\leq \const_{\nu,p,a}rt^{p-1}\\
+\const_prt^{q-1}\left|\int_{D(t)\setminus D(r_0)} \Im \frac{1}{z^q}\dd \nu^{\overline \up} (z)\right|
\; \text{при всех $|t|\geq t_0$, $r\in [0,1]$}.
\end{multline*}
Отсюда и из условия Бляшке \eqref{cBq} 
следует \eqref{trnua}.
\end{proof}
\end{corollary}

\begin{theorem}[{\rm ср. с \cite[Лемма 2.1.2]{KhDD92}}]\label{th14} Пусть $\nu\in \mathcal M(\CC)$ --- заряд конечного типа при порядке $p\in \RR^+$ около $\infty$, т.\,е. $\type_p^{\infty}[\nu]{<}+\infty$, $q:=[p]$ --- целая часть $p$,  $0\notin \supp \nu$, т.\,е. для некоторого $r_0>0$ выполнено \eqref{Dr0}.
\begin{enumerate}[{\rm 1.}]
\item\label{c:0nuniii} Тогда существуют выметание $\nu^{\bal[q]}$ из $\CC^{\up}$ и  постоянная $C:=\const_{\nu,p}$,
для  которых при всех $r\in \RR^+$ справедлива оценка
\begin{equation}\label{es:pob}
|\nu^{\bal[q]}|^{\rad}(r)\leq C \left(r^p+r^q\Bigl|\int_{D(r)\setminus D(r_0)} \Im \frac{1}{z^q}\dd \nu^{{\up}} (z)\Bigr|\right).
\end{equation}

\item\label{c:0nuii} При любом $p\in \RR^+$
 выметание $\nu^{\bal[q]}$  удовлетворяет условию
\begin{equation}\label{es:plog}
|\nu^{\bal[p]}|^{\rad}(r)=O(r^p\log  r)\quad \text{при $r\to +\infty$}.
\end{equation}

\item\label{c:0nuni} Если  $p\in \RR^+\setminus \NN_0$ --- нецелое число,  то 
выметание $\nu^{\bal[q]}$ из $\CC^{\up}$ конечного типа при том же порядке $p$, т.\,е.  $\type_{p}^{\infty}\bigl[\nu^{\bal[q]}\bigr]<+\infty$.
 
\item\label{c:0nuiii} Если $p\in \NN_0$, --- при этом, конечно, $q=p$, --- и 
для $\nu$ в  верхней полуплоскости выполнено условие Бляшке  рода $p=q$ около $\infty$ вида \eqref{cBq}, 
 то
 выметание $\nu^{\bal[q]}=\nu^{\bal[p]}$ из $\CC^{\up}$ конечного типа при порядке $p$.
\end{enumerate}
\end{theorem}

\begin{proof} \ref{c:0nuniii}. Применим оценку  \eqref{nu:estTosq} теоремы \ref{th:cupq} с $0<t=r$ и $a=1/2$ в \eqref{0rt0q}:
\begin{multline*}
\bigl|\nu^{\bal[q]}\bigr|^{\RR}\bigl([0,2t]\bigr)\leq \const_{\nu,p,a}rt^{p-1}+ |\nu^{\overline{\up}}|(t,r) 
\\
+\frac{2r}{\pi}\biggl(\int_r^{at}\frac{|\nu^{\overline{\up}}|(t,s)}{s^2}\dd s\biggl)^+
+\const_prt^{q-1}\left|\int_{D(t)\setminus D(r_0)} \Im \frac{1}{z^q}\dd \nu^{\overline \up} (z)\right|
\\
\leq \const_{\nu,p} t^p+\const_pt^q\left|\int_{D(t)\setminus D(r_0)} \Im \frac{1}{z^q}\dd \nu^{\overline \up} (z)\right|.
\end{multline*}
Такую же  оценку имеем и для $\bigl|\nu^{\bal[q]}\bigr|^{\RR}\bigl([-2t,0]\bigr)$, что доказывает 
\eqref{es:pob}. 

\ref{c:0nuii}. При произвольном  $p$ соотношение \eqref{es:plog} получаем из 
\begin{equation}\label{e:pq}
|\nu^{\bal[q]}|^{\rad}(r)\leq C r^p+Cr^q\int_{r_0}^r \frac{1}{t^q}\dd \nu^{\rad} (t)
\end{equation} 
интегрированием по частям последнего интеграла.  

\ref{c:0nuni}. Для  $p\notin \NN_0$ интегрирование по частям \eqref{e:pq} дает  $\type_{p}^{\infty}\bigl[\nu^{\bal[q]}\bigr]<+\infty$.

\ref{c:0nuiii}. При целом $p=q$ из  \eqref{cBq} интеграл в \eqref{es:pob} --- это  $O(1)$ при $r\to +\infty$, что  дает $|\nu^{\bal[q]}|^{\rad}(r)\leq C \bigl(r^p+O(1)r^q\bigr)=O(r^p)$, $r\to +\infty$. 
\end{proof}
 \subsection{Выметание конечного рода и повторные интегралы}\label{ss:povi}
 Здесь мы используем теорему о повторных интегралах, приведённую в \cite[5.1]{KhI}, в частном случае, когда $k=p=2$, $\RR^2$ отождествлено с $\CC$ и
 $Z=X=\CC_*:=\CC\setminus \{0\}$, а в роли семейства $T$ из \cite[(5.2)]{KhI}  рассматривается семейство выметаний рода $q$ мер Дирака  $\delta_z$ из верхней полуплоскости $\CC^{\up}$ на $\CC_{\overline{\lw}}$.
Для всех точек  $z\in \CC_{\overline{\lw}}$ можно полагать  заряд $\Omega^{[q]}(z,\cdot\,; \CC^{\up})$ рода $q$ равным мере Дирака $\delta_z$. Тогда, используя понятие интеграла семейства мер по заряду из \cite[\S~5, 5.1]{KhI}, можем записать равенство \eqref{df:bal} из определения \ref{df:Bal} в более сжатой форме
 \begin{equation}\label{df:bald}
    \nu^{\bal [q]}(B):=\int_{\CC} \Omega^{[q]}(z, B)\dd \nu(z), \quad B\subset \mathcal B_{\rm b} (\CC).
  \end{equation} 
Версия  \cite[теорема 6]{KhI} для выметания заряда конечного рода $q$ из $\CC^{\up}$ ---
 \begin{theorem}\label{th:IB} Пусть  $\delta_z^{\bal[q]}$ и $\nu^{\bal[q]}$ --- выметания из $\CC^{\up}$на $\CC_{\overline{\lw}}$ соответственно меры Дирака\/ $\delta_z$ и заряда\/ $\nu\in \mathcal M(\CC_*)$. Тогда
\begin{enumerate}[{\rm (i)}] \item\label{dei} для любого  $B\in \mathcal{B}(\CC)$ с 
$ \boldsymbol{1}_B(z):=1$ при $z\in B$ и  $\boldsymbol{1}_B(z):=0$  при $z\notin B$
\begin{equation}\label{1balo} 
\int \mathbf{1}_B (z')\dd \delta_z^{\bal[q]}(z')=\Omega^{[q]}(z,B; \CC^{\up})
 \quad\text{при всех $z\in \CC$}, 
 \end{equation}
 т.\,е. $\delta_z^{\bal[q]}=\Omega^{[q]}(z,\cdot\,; \CC^{\up})$ для всех $z\in \CC$; в частности, при $z\in \CC^{\up}$ для функции распределения\footnote{Обозначение для функции распределения заряда на $\RR$ из \eqref{nuR}.} $\bigl(\delta_{z}^{\bal[q]}\bigr)^{\RR}$ заряда $\delta_{z}^{\bal[q]}$, сосредоточенного на $\RR$, через дифференциалы $\dd$ от неё имеют место равенства
\begin{equation}\label{omOm}
\dd \,\bigl(\delta_{z}^{\bal[q]}\bigr)^{\RR}(t)=\dd\, \bigl(\Omega^{[q]}(z, \cdot)\bigr)^{\RR}(t)\overset{\eqref{df:kP+}}{=}\Poi^{[q]} (t,z) \dd t, \quad z\in \CC^{\up};
\end{equation}

\item\label{deiii} выметание $\nu^{\bal[q]}=(\nu^+)^{\bal[q]}-(\nu^-)^{\bal[q]}$ --- разность
выметаний из  $\CC^{\up}$ верхней и нижней вариаций $\nu$, равных интегралу семейства мер $\delta_z^{\bal[q]}$ по $\nu^{\pm}$, т.\,е.
имеют место равенства\/ {\rm (см. \cite[5.1, (5.3)]{KhI})}
\begin{equation}\label{eqomd}
(\nu^{\pm})^{\bal[q]} =\int \delta_z^{\bal[q]} \dd \nu^{\pm}(z)= \int \Omega^{[q]}(z,\cdot\,; \CC^{\up}) \dd \nu^{\pm}(z); 
\end{equation}

\item\label{deii} для любой непрерывной финитной функции $f\in C_0(\CC_*)$ определена непрерывная на $\CC_*$ функция $z\mapsto \delta_z^{\bal[q]}(f)$, $z\in \CC_*$; 

\item\label{deiv} 
для любой $\nu^{\bal[q]}$-интегрируемой функции $F\colon \CC_*\to \RR_{\pm\infty}$ 
\begin{subequations}\label{c:del} 
\begin{align} 
\int F(z)\dd \nu^{\bal[q]}(z)&=\int \bigl(\mathcal P_{\CC^{\up}}^{[q]}F\bigr)(z)\dd
\nu (z), \tag{\ref{c:del}a}\label{{c:del}a}\\ \intertext{где подынтегральная функция в правой части \eqref{{c:del}a}} 
\bigl(\mathcal P_{\CC^{\up}}^{[q]}F\bigr)(z)&:=\int F(z')\dd \delta_z^{\bal[q]}(z')
 \tag{\ref{c:del}b}\label{{c:del}b} 
\end{align}
\end{subequations} $\nu$-интегрируема и называется далее интегралом Пуассона рода $q$ функции $F$ на  $\CC^{\up}$\/ {\rm (ср. с \cite[3.1, (3.2)]{KhI})}.
\end{enumerate} 
\end{theorem}
\begin{proof} \eqref{dei}. Равенство \eqref{1balo} следует из очевидных равенств 
\begin{equation*} 
\Omega_{\CC^{\up}}^{[q]}(z,B)=\int \Omega_{\CC^{\up}}^{[q]}(z', B) \dd \delta_z(z')
\overset{\eqref{df:bald}}{=}\delta_z^{\bal[q]}(B)=\int \mathbf{1}_B (z')\dd \delta_z^{\bal[q]}(z'). 
\end{equation*} 
Это сразу даёт первое  равенство в \eqref{omOm}, а второе в \eqref{omOm} --- из   \eqref{{Omt1t2}a}.

\eqref{deiii}. Равенства \eqref{eqomd} п.~\eqref{deiii} сразу следуют из равенства \eqref{1balo} п.~\eqref{dei} по определению \ref{df:Bal} в форме \eqref{df:bal}$\Leftrightarrow$\eqref{df:bald}  после интегрирования обеих частей равенства \eqref{1balo} по верхней и нижней вариациям $\nu^{\pm}$.

\eqref{deii}.   	$f\in C_0(\CC_*)$ доопределим по непрерывности $f(0)=0$ и  положим 
 	\begin{equation}\label{dfFF} 
F(z):=\int f(z')\dd \delta_z^{\bal[q]}(z'),  \quad z\in \CC_*.
 	\end{equation} 
 	Необходимо доказать непрерывность функции $F$. На $\CC_{\overline \lw}$ это очевидно, поскольку для  точек $z\in \CC_{\overline \lw}$ имеем $\delta_z^{\bal[q]}=\delta_z$ и $F(z)=f(z)$. В случае $z\in \CC^{\up}$ интеграл из \eqref{dfFF} согласно п.~\eqref{dei} и тождеству \eqref{dqt}
предложения \ref{pr:b} расписывается в виде
\begin{multline}\label{dRt}
F(z)=\int_{\RR}f(t) \dd\, \bigl(\delta_z^{\bal [q]}\bigr)^{\RR}(t)
\overset{\eqref{dqt}}{=}\int_{\RR}f(t) \dd\, (\delta_z^{\bal})^{\RR}(t)
\\+\frac{1}{\pi}\sum_{k=1}^q\left(\int_{\CC^{\up}} \Im \frac{1}{(z')^k}\dd \delta_z (z')\right)\, \frac{1}{k}\int_{\RR}f(t)t^{k}\dd t
\\
=\mathcal P_{\CC^{\up}} f(z)+\frac{1}{\pi}\sum_{k=1}^q\Im \frac{1}{kz^k}\, \int_{\RR}f(t)t^{k}\dd t,
\end{multline}
где  первое слагаемое справа --- гармоническое продолжение посредством классического интеграла Пуассона функции $f$ в верхнюю полуплоскостью \cite[3.1, теорема 6(i)]{KhI}, которое по известным свойствам дает непрерывную функцию на $\CC^{\overline \up}$, совпадающую с $f$. Сумма в правой части  \eqref{dRt}  из ее явного вида непрерывна на $\CC^{\up}_*:=\CC^{\up}\setminus \{0\}$ и стремиться к нулю при приближении к точкам на $\RR_*$ из $\CC^{\up}$. Непрерывность $F$ из \eqref{dfFF} доказана.

\eqref{deiv}. Наконец, условия (1)--(2) 
теоремы о повторных интегралах из \cite[5.1]{KhI} в  	случае, когда в роли мер $\tau_z$ из 
\cite[(5.2)]{KhI} выступают выметания $\delta_z^{\bal[q]}$ рода $q$ мер Дирака, несколько трудоемко проверяются непосредственно, исходя из определений $\nu$-измеримости. Вместе с \eqref{deii} по 	теореме о повторных интегралах из равенства \cite[(5.4)]{KhI}, где на роль зарядов  $\nu$ и $\mu$ выбираем соответственно заряды $\nu^{\bal[q]}$ и $\nu$,  получаем  \eqref{c:del} из \eqref{deiv}.
\end{proof}

\subsection{Выметание конечного рода заряда из угла}\label{ss:ang} 
В этом пункте результаты о выметании конечного рода заряда из полуплоскости адаптируются для выметания заряда, заданного на всей плоскости $\CC$, из угла 
\begin{equation}\label{abr}
\angle(\alpha, \beta) \subset \CC, \quad -\infty <\alpha <\beta\leq \alpha +2\pi, \quad \text{раствора $\beta-\alpha$.}
\end{equation}
Редукция угла $\angle(\alpha , \beta)$  к верхней полуплоскости с помощью конформной замены переменных 
\begin{equation}\label{tildez} 
\tilde z:=(ze^{-i\alpha})^{\frac{\pi}{\beta-\alpha}}, \quad z\in \angle\,[\alpha, \beta],
\end{equation} 
где при  возведении в степень $\frac{\pi}{\beta-\alpha}$ рассматривается аналитическая ветвь,
положительная на  $\RR^+$, позволяет определить гармонический заряд рода $q$ для  угла $\angle(\alpha,\beta)$. Конкретнее, пусть $z\in
\angle (\alpha, \beta)$ и $\tilde B$ --- образ пересечения 
множества $B$ с двумя ограничивающими угол $\angle(\alpha,\beta)$ лучами при редукции угла
$\angle (\alpha, \beta)$ к верхней полуплоскости. Тогда, с учётом  конформной инвариантности гармонической меры для угла $\angle (\alpha, \beta)$ (ср. с \cite[(4.23)]{KhI}), {\it гармонический заряд рода $q$ для угла $\angle (\alpha, \beta)$} --- это  функция 
$\Omega_{\angle(\alpha,\beta)}^{[q]}\colon \angle(\alpha,\beta)\times \mathcal B_{\rm b} (\CC) \to \RR$,
определенная как 
\begin{subequations}\label{se:HC+}
\begin{align} 
&\Omega_{\angle(\alpha,\beta)}^{[q]}(z,B):=
\omega_{\CC^{\up}}(\tilde z,\tilde B)+\frac{1}{\pi} 
 \int_{\tilde B\cap \RR_{\pm\infty}}\Im \frac{t^q-\tilde z^q}{\tilde z^{q}(t-\tilde z)} \dd \lambda_{\RR}(t) 
\tag{\ref{se:HC+}a}\label{seHC+:b}
\\ 
&=
\omega(\tilde z,\tilde B)+\frac{1}{\pi} 
 \sum_{k=1}^{q} \Im \frac{1}{\tilde z^{k}} \,\int_{\tilde B\cap \RR_{\pm\infty}} t^{k-1}\dd \lambda_{\RR}(t),\quad z\in \angle(\alpha,\beta),
\tag{\ref{se:HC+}b}\label{seHC+:d}
\end{align}
\end{subequations}
где для \eqref{seHC+:b} иногда удобно использовать обозначение $\Omega^{[q]}(z,B; \angle(\alpha,\beta))$.
В соответствии с \eqref{df:Bal} для $q\in \NN_0$ и  $\nu \in \mathcal M(\CC)$ 
{\it выметание рода\/ $q$ заряда\/ $\nu$ из угла $\angle(\alpha, \beta)$ 
на $\CC\setminus \angle (\alpha,\beta)$} определяется  равенством 
  \begin{multline}\label{df:bal+}
    \nu_{\CC\setminus \angle(\alpha,\beta)}^{\bal [q]}(B):=\int_{\angle(\alpha,\beta)} \Omega_{\angle(\alpha,\beta)}^{[q]}(z, B\cap\angle(\alpha,\beta))\dd \nu(z)\\+
\nu\bigl(B\cap (\CC\setminus \angle(\alpha,\beta) \bigr), \quad B\subset \mathcal B_b(\CC).
  \end{multline} 
В частности, для меры Дирака $\delta_z$ при $z\in \angle (\alpha,\beta)$ согласно 
равенствам \eqref{omOm} теоремы  \ref{th:IB} имеем равенства 
\begin{equation*}
\dd \,\bigl(\delta_{z}^{\bal[q]}\bigr)^{\RR}(t)=\dd\, \bigl(\Omega^{[q]}(z, \cdot)\bigr)^{\RR}(t)\overset{\eqref{df:kP+}}{=}\Poi^{[q]} (t,z) \dd t, \quad z\in \CC^{\up}.
\end{equation*}

Выметание рода $q$ не меняет часть заряда, сосредоточенную в  $\CC\setminus \angle(\alpha,\beta)$. Утверждения предшествующих подразделов \ref{sss:mu}--\ref{ss:povi} о выметании конечного рода из верхней полуплоскости путём довольно громоздких замен переменных переносятся и на выметание конечного рода из угла. Мы не приводим все аналоги этих утверждений, а  ограничимся только переносом теоремы \ref{th14} на углы $\angle (\alpha, \beta)$ из \eqref{abr}, который дает 

\begin{theorem}[{\rm ср. с \cite[Лемма 2.1.2]{KhDD92}}]
\label{th15} Пусть $\nu\in \mathcal M(\CC)$ --- заряд конечного типа при конечном порядке $p\in \RR^+$ около $\infty$, т.\,е. $\type_p^{\infty}[\nu]{<}+\infty$, 
\begin{equation}\label{dgen}
q:=\Bigl[\frac{\beta-\alpha}{\pi}\,p\Bigr] \text{--- целая часть числа $\frac{\beta-\alpha}{\pi}p$,}
\end{equation} 
$0\notin \supp \nu$, т.\,е. для некоторого $r_0>0$ выполнено \eqref{Dr0}.
\begin{enumerate}[{\rm 1.}]
\item\label{c:0nuniii+} Тогда существуют выметание $\nu_{\CC\setminus \angle(\alpha,\beta)}^{\bal[q]}$ из угла $\angle(\alpha,\beta)$ и  постоянная $C:=\const_{\nu,p, \alpha,\beta}$,
для  которых при всех $r\in \RR^+$ 
\begin{multline}\label{es:pob+}
\bigl|\nu_{\CC\setminus \angle(\alpha,\beta)}^{\bal[q]}\bigr|^{\rad}(r)
\\
\leq C \left(r^p+r^\frac{\pi q}{\beta-\alpha}\Biggl|\;\int\limits_{(D(r)\setminus D(r_0))\cap  \angle(\alpha,\beta} \Im \frac{1}{(ze^{-i\alpha})^\frac{\pi q}{\beta-\alpha}}\dd \nu (z)\Biggr|\right).
\end{multline}

\item\label{c:0nui} При любых $p$  выметание $\nu_{\CC\setminus \angle(\alpha,\beta)}^{\bal[q]}$ из $\angle(\alpha,\beta)$ удовлетворяет условию 
\begin{equation}\label{es:plog+}
\bigl|\nu_{\CC\setminus \angle(\alpha,\beta)}^{\bal[q]}\bigr|^{\rad}(r)=O(r^p\log  r)\quad \text{при $r\to +\infty$}.
\end{equation}

\item\label{c:0nuni+} Если  $\frac{\beta-\alpha}{\pi}p\in \RR^+\setminus \NN_0$ --- нецелое число,  то 
выметание $\nu_{\CC\setminus \angle(\alpha,\beta)}^{\bal[q]}$ из $\angle(\alpha,\beta)$ конечного типа при том же порядке $p$, т.\,е.  $\type_{p}^{\infty}\bigl[\nu_{\CC\setminus \angle(\alpha,\beta)}^{\bal[q]}\bigr]<+\infty$.  
 
\item\label{c:0nuiii+} Если  $q=\frac{\beta-\alpha}{\pi}p\in \NN_0$ --- целое число, --- в частности, $\frac{\pi q}{\beta-\alpha}=p$, --- и для $\nu$ в угле $\angle(\alpha,\beta)$
 выполнено условие Бляшке рода $p$ (около $\infty$), а именно:
\begin{multline}\label{ka+}
\Biggl|\;\int\limits_{(D(r)\setminus D(r_0))\cap \angle(\alpha,\beta)} \Im \frac{1}{(ze^{-i\alpha})^\frac{\pi q}{\beta-\alpha}}\dd \nu (z)\Biggr|\\
=\Biggl|\;\int\limits_{(D(r)\setminus D(r_0))\cap \angle(\alpha,\beta)} \Im \frac{1}{(ze^{-i\alpha})^p}\dd \nu (z)\Biggr|=O(1), \quad r\to +\infty,
\end{multline}
 то  выметание $\nu_{\CC\setminus \angle(\alpha,\beta)}^{\bal[q]}$ из угла $\angle(\alpha,\beta)$  конечного типа при порядке $p$.
\end{enumerate}
\end{theorem}
\begin{remark}\label{remcl}
Очевидно, классическое условие Бляшке для заряда $\nu$ из
\cite[определение 4.6]{KhI} --- это условие Бляшке рода $1$ по определению \eqref{ka+}.
\end{remark}
\begin{remark}\label{remA1} При $\alpha =0$ и $\beta=2\pi$ из теоремы \ref{th15} сразу получаем теорему A1 из подраздела \ref{singleray}, п.~\ref{b1ray}, о выметание заряда  $\nu\in \mathcal M(\CC)$ конечного типа при порядке $p$ на $\RR^+$ из $\CC\setminus \RR^+$.  

\end{remark}
\begin{remark}\label{remRay}
Для произвольного заряда (меры) $\nu\in \mathcal M(\CC)$ конечного типа  при порядке $p$, исходя из классического выметания рода $0$ из первой части \cite{KhI} для допустимой системы лучей, в которой раствор каждого дополнительного угла менее $\pi/p$, и конечного числа выметаний рода $q$ из \eqref{dgen} для каждого дополнительного к системе лучей угла раствора не менее $\pi/p$ можно построить {\it глобальное выметание $\nu_{S}^{{\Bal}}$  на  любую замкнутую систему лучей\/ $S$ c началом в нуле} по простой схеме:{\it
\begin{enumerate}[{\rm (i)}]
\item\label{bSi} для некоторого произвольного $r_0\in \RR^+_*$ представить $\nu$ в виде суммы 
\begin{equation}\label{nui0}
\nu=\nu\bigm|_{D(r_0)}+\nu\bigm|_{\CC \setminus  D(r_0)}=:\nu_0+\nu_{\infty};
\end{equation}
\item\label{bSii} построить классическое выметание рода $0$ заряда $\nu_0$ на систему лучей $S$ \cite[теоремы 2--4]{KhI};
\item\label{bSiii} произвести классическое выметание рода $0$ одновременно из всех дополнительных к $S$ углов раствора $<\pi/p$ и из конечного числа дополнительных к $S$ углов \eqref{abr},  раствора $\geq \pi/p$, в которых полная вариация $|\nu|$ заряда $\nu$ удовлетворяет классическому условию Бляшке рода $1$ около $\infty$ в угле $\angle(\alpha,\beta)$ \cite[теоремы 2--4, определение 4.6]{KhI};  
\item\label{bSiv} в, возможно,  оставшемся конечном числе дополнительных углов \eqref{abr} раствора $\geq \pi/p$, в которых не выполнено условие Бляшке рода $1$, построить выметание \eqref{df:bal+} рода  $q$, выбранного как в \eqref{dgen}.
\end{enumerate}}
Последовательность применения пунктов \eqref{bSii}--\eqref{bSiv} здесь  можно поменять.
При этом в результате действий по пунктам \eqref{bSi}--\eqref{bSiii} выметание будет оставаться зарядом (соответственно мерой) конечного типа при порядке $p$, а при необходимости применения п.~\eqref{bSiv}
результатом будет заряд и  возможны вариации его роста в соответствии с соотношениями \eqref{es:pob+}, \eqref{es:plog+} и пунктами \ref{c:0nuniii+}--\ref{c:0nuiii+}  теоремы \ref{th15}.
\end{remark}

\section{Выметание конечного рода $\delta$-субгармонической функции}\label{bqsf} 
\setcounter{equation}{0}

\subsection{Выметание $\delta$-субгармонической функции конечного порядка из верхней полуплоскости}\label{bduus}

\begin{definition}\label{df:Bv}
Пусть $v\in \dsbh_*(\CC)$. Функцию $v^{{\Bal}}\in \dsbh_*(\CC)$ называем {\it выметанием функции\/ $v$ из верхней полуплоскости\/ $\CC^{\up}$ на $\CC_{\overline \lw}$}, если $v^{{\Bal}}=v$ на $\CC_{\overline \lw}$ вне некоторого полярного множества и  ее сужение $v^{{\Bal}}\bigm|_{\CC^{\up}}$ на $\CC^{\up}$ --- гармоническая функция. 

Выметание $v^{{\Bal}}$, вообще говоря, не единственно. 
\end{definition}

\begin{theorem}\label{thdsb}  Пусть $v\in \dsbh_*(\CC)$, 
с  зарядом Рисса $\nu:=\nu_v\in \mathcal M(\CC)$ конечного типа при порядке $p\in \RR_*^+$, $q:=[p]$.  Для произвольного $r_0\in \RR_*^+$ представим заряд Рисса $\nu$ 
в виде суммы зарядов 
\begin{equation}\label{nui}
\nu=\nu\bigm|_{D(r_0)\cap \CC^{\up}}+\nu\bigm|_{\CC \setminus  (D(r_0)\cap \CC^{\up})}=:\nu_0+\nu_{\infty}.
\end{equation}
Тогда существует  выметание $v^{{\Bal}}$ с зарядом Рисса $\nu_0^{\bal[0]}+ \nu_{\infty}^{\bal[q]}$, представимое в виде  $v^{{\Bal}}=v_+-v_-+H$, где $H\in \Har(\CC)$, обе функции  $v_{\pm}\in \sbh(\CC)$ порядка $p$ гармоничны в $\CC^{\up}$, а также 
\begin{enumerate}[{\rm (i)}]
\item\label{vvii}   при произвольном   $p\in \RR_*^+$ выполнены  соотношения
\begin{equation}\label{llogv}
v_\pm (z)\leq O\bigl(|z|^p\log^2 |z|\bigr), \quad z\to \infty.
\end{equation}

\item\label{vvi}  при нецелом $p$ функции  $v_{\pm}$  конечного типа при порядке $p$;

\item\label{vviiA}  
Если  $p\in \NN_0$ 
и для $\nu$ в угле $\angle(0,\pi)$
выполнено условие Бляшке  рода $p$ \eqref{ka+}, то имеют место два соотношения 
\begin{equation}\label{llogvA}
v_\pm (z)\leq O\bigl(|z|^p\log |z|\bigr), \quad z\to \infty.
\end{equation}
\end{enumerate}
Если функция $v$ порядка не выше $p$, т.\,е. $\ord_{\infty}[v]\leq p$ в обозначениях из  \cite[(2.1o)]{KhI}, то в качестве функции  $H\in \Har (\CC)$ можно выбрать гармонический многочлен степени не выше $q$.
\end{theorem} 
\begin{proof} Можем  считать, что $\bigl(D(r_0)\cap \CC^{\up}\bigr)\cap\, \supp \nu=\varnothing$, т.\,е. $\nu_0=0$, поскольку возможность классического выметания $v_{\nu_0}^{\bal[0]}$ из $\CC^{\up}$ с мерой Рисса   $\nu_0^{\bal[0]}$  установлена в  \cite[теорема 8]{KhI} для функций вида
\begin{equation*}
v_{\nu_0}(z):= \int\limits_{D(r_0)\cap \CC^{\up})} \log |z-\zeta |\dd \nu_0(\zeta), \; z\in \CC,\quad \text{с мерой Рисса $\nu_0$}.
\end{equation*}
Таким образом, далее допустимо рассматривать в \eqref{nui} случай $\nu=\nu_{\infty}$. Кроме того, достаточно рассмотреть ситуацию, когда $v\in \sbh_*(\CC)$ --- функция с мерой Рисса $\nu\in \mathcal M^+(\CC)$
конечного типа при порядке $p$
 и задана представлением Вейерштрасса\,--\,Адамара \cite[6.1, (6.6)--(6.9)]{KhI}, \cite[4.2]{HK}:
\begin{equation}\label{WArep}
v(z)=\int K_q(\zeta,z) \dd \nu(\zeta), \; z\in \CC, \quad
K_q(\zeta,z):=\log \Bigl|1-\frac{z}{\zeta}\Bigr|+\sum_{k=1}^{q}\Re \frac{z^{q}}{q{\zeta}^q}\,.
\end{equation}
Необходимо доказать равенство
\begin{equation}\label{WArep+}
\int K_q(\zeta,z) \dd \nu(\zeta)=\int K_q(\zeta,z) \dd \nu^{\bal[q]}(\zeta) \quad\text{для всех  $z\notin \CC^{\up}$}.
\end{equation}
По определению \ref{df:Bal} согласно равенствам \eqref{df:bal} и \eqref{vnvbal}  можно рассматривать только   меру $\nu\in \mathcal M^+(\CC^{\up})$,  сосредоточенную  в $\CC^{\up}$, поскольку выметание рода $q$ не меняет часть заряда, сосредоточенную в $\CC_{\overline{\lw}}$. Если в обозначении  $\delta_{\zeta}$ для меры Дирака в точке $\zeta\in \CC^{\up}$ будет доказано равенство
\begin{equation}\label{Kqdb}
K_q(\zeta,z)=\int K_q(\xi, z) \dd \delta_{\zeta}^{\bal[q]}(\xi)\quad\text{для всех $z\in \CC_{\overline \lw}$, 
$\zeta\in \CC^{\up}$},
\end{equation}
то справедливо \eqref{WArep+},  поскольку из \eqref{Kqdb} по теореме \ref{th:IB}\eqref{deiv}
из равенств  \eqref{c:del} с $F(\zeta)=K_q(\zeta, z)$ при отмеченных  соглашениях о мере $\nu$ 
с учётом \eqref{Dr0q} имеем цепочку равенств 
\begin{multline*}
\int K_q(\zeta,z) \dd \nu^{\bal[q]}(\zeta) \overset{\eqref{{c:del}a}}{=}
\int \bigl({\mathcal P}_{\CC^{\up}}^{[q]} K_q(\cdot, z)\bigr) (\zeta)\dd \nu (\zeta)
\\
\hspace{-4mm}\overset{\eqref{{c:del}b}}{=}\int \int K_q(\xi, z) \dd \delta_{\zeta}^{\bal[q]}(\xi) \dd \nu (\zeta)
\overset{\eqref{Kqdb}}{=}\int K_q(\zeta,z) \dd \nu(\zeta) \text{ для всех $z\notin \CC^{\up}$.}
\end{multline*}
Итак, для завершения доказательства достаточна 
\begin{lemma}\label{lKqd}  Имеет место\/ \eqref{Kqdb}.
\end{lemma}
\begin{proof}[Доказательство леммы \ref{lKqd}]
Можно расписать правую часть \eqref{Kqdb}, исходя из тождества \eqref{dqt} предложения \ref{pr:b}, c помощью функции распределения $\bigl(\delta_{\zeta}^{\bal[q]}\bigr)^{\RR}$ заряда $\delta_{\zeta}^{\bal[q]}$, сосредоточенного на $\RR$, в виде
\begin{multline}\label{dqtd}
\hspace{-3mm}\int K_q(\xi, z) \dd \,\delta_{\zeta}^{\bal[q]}(\xi)=
\int_{\RR} K_q(t, z) \dd \,\bigl(\delta_{\zeta}^{\bal [q]}\bigr)^{\RR}(t)
\overset{\eqref{omOm}}{=}\int_{\RR} K_q(t, z) \Poi^{[q]} (t,\zeta) \dd t
\\
\overset{\eqref{df:kPq}}{=}\int_{\RR} K_q(t, z) \,\frac{1}{\pi} \, \Im\frac{t^{q}}{ {\zeta}^{q}(t-\zeta)}
 \dd t, \quad  z\notin \CC^{\up}, \; \zeta \in \CC^{\up}. 
\end{multline}
Рассмотрим комплексификации подынтегральных выражений. Полагаем 
 \begin{equation}\label{CKP}
\mathcal K_q(w,z)\overset{\eqref{WArep}}{=}
\log \Bigl(1-\frac{z}{w}\Bigr)+\sum_{k=1}^{q} \frac{z^{q}}{q{w}^q}, \quad \CC^{\overline{\up}}\ni w\neq z \in \CC_{\overline{\lw}}, 
\end{equation}
где для функции $w\mapsto \log(1-z/w)$ выбрана одна из аналитических ветвей в $\CC^{\up}$, что возможно ввиду  
$z/w\neq 1$ в \eqref{CKP}. Из разложения в ряд Тейлора этой логарифмической функции при  $|w|>|z|$ получаем \cite[лемма 4.2]{HK}
\begin{equation}\label{razK}
\mathcal K_q(w,z)=O\bigl(|w|^{-q-1}\bigr) \quad \text{при $\CC^{\overline{\up}}\ni w\to \infty$}.
\end{equation}
Кроме того,  из вида функции \eqref{CKP} сразу следует 
\begin{equation}\label{raz0}
\mathcal K_q(w,z)=O\bigl(|w|^{-q}\bigr), \quad \text{при $\CC^{\overline{\up}}\ni w\to 0$},
\end{equation}
а также возможна одна логарифмическая особенность на $\RR$ при  $w=z\in \RR$.
Следующий интеграл по границе верхнего полукруга $\partial D^{\up} (R)$, где $D^{\up}(R):=D(R)\cap \CC^{\up}$,  $R>|w|$, обозначаемый как
\begin{equation}\label{Int1}
\mathcal I_q(\zeta,z):=\int_{\partial D^{\up} (R)} \mathcal K_q(w,z) \,\frac{w^q}{2\pi i \zeta^q}\, \frac{\dd w}{w-\zeta}\,,
\end{equation}
абсолютно сходится при каждом $\zeta \notin \RR$, так как подынтегральное выражение в нём непрерывно всюду на $\RR$, кроме возможной лишь логарифмической особенности  на $\RR$ при  $w=z\in \RR$, а  в нуле существует предел 
\begin{equation*}
\lim_{w\to 0} \mathcal K_q(w,z)\, \frac{w^q}{2\pi i \zeta^q}\,\frac{1}{w-\zeta}
\overset{\eqref{raz0},\eqref{CKP}}{=}-\frac{z^q}{2\pi i \zeta^{q+1}}\,.
\end{equation*}
При $\zeta\in \CC^{\up}$ по теореме о вычетах при $R>|w|$ имеем
\begin{equation}\label{IK}
\mathcal I_q(\zeta, z)=\mathcal K_q(\zeta,z)\, \frac{\zeta^q}{\zeta^q}=\mathcal K_q(\zeta,z), 
\end{equation}
а при замене $\zeta\in \CC^{\up}$ на сопряжённое $\bar \zeta\in \CC_{\lw}$ в \eqref{Int1} ввиду голоморфности подынтегрального выражения в $D^{\up}$ получаем
$\mathcal I_q(\bar \zeta, z)=0$. Отсюда, вычитая последнее равенство из \eqref{IK}, из вида $\mathcal I_q$ в \eqref{Int1} имеем
\begin{equation}\label{InKv}
\int_{\partial D^{\up} (R)} \mathcal K_q(w,z)
\frac{w^q}{2\pi i}\left(\frac{1}{\zeta^q(w-\zeta)}-\frac{1}{{\bar\zeta}^q(w-\bar\zeta \,)}\right) \dd w
=\mathcal K_q(\zeta,z).
\end{equation} 
При стремлении $R\to +\infty$ часть интеграла \eqref{InKv} по верхней	 полуокружности 
$\CC^{\up}\cap \partial D(R)$ стремится к нулю, поскольку  ввиду   \eqref{razK} подынтегральное выражение в этой части интеграла  ведёт себя как $O(R^{-2})$ при $R\to +\infty$. Таким образом, из 
\eqref{InKv} следует равенство 
\begin{equation*}
\int_{\RR} \mathcal K_q(t,z)\,
\frac{t^q}{2\pi i}\left(\frac{1}{\zeta^q(t-\zeta)}-\frac{1}{{\bar\zeta}^q(t-\bar\zeta \,)}\right) \dd t
=\mathcal K_q(\zeta,z), \quad \zeta\in \CC^{\up}, z\in \CC_{\overline \lw}.
\end{equation*}
Здесь часть подынтегрального выражения 
\begin{equation*}
\frac{t^q}{2\pi i}\left(\frac{1}{\zeta^q(t-\zeta)}-\frac{1}{{\bar\zeta}^q(t-\bar\zeta \,)}\right)=
\frac{1}{\pi}\,\Im \frac{t^q}{\zeta^q(t-\zeta)}\,, \quad t\in \RR,
\end{equation*}
всегда вещественна и последний интеграл можно переписать в виде
\begin{equation*}
\int_{\RR} \mathcal K_q(t,z)\, \frac{1}{\pi}\,\Im \frac{t^q}{\zeta^q(t-\zeta)} \dd t
=\mathcal K_q(\zeta,z), \quad \zeta\in \CC^{\up}, z\in \CC_{\overline \lw}.
\end{equation*}
Для вещественных частей обеих сторон этого равенства имеем  
\begin{equation*}
\int_{\RR} \Re  \mathcal K_q(t,z)\cdot \frac{1}{\pi}\,\Im \frac{t^q}{\zeta^q(t-\zeta)} \dd t
=\Re \mathcal K_q(\zeta,z), \quad \zeta\in \CC^{\up}, z\in \CC_{\overline \lw}.
\end{equation*}
Поскольку $\Re  \mathcal K_q(t,z)\overset{\eqref{WArep},\eqref{CKP}}{=}K_q(t,z)$ при  $t\in \RR$, то отсюда ввиду \eqref{dqtd} получаем требуемое 
равенство \eqref{Kqdb} и лемма доказана. 
\end{proof}
Из доказанного имеем часть теоремы  \ref{thdsb} до п.~\eqref{vvii}--\eqref{vvi}, где две 
функции $v_{\pm}\in \sbh(\CC)$ с мерами Рисса $(\nu^{\bal[q]})^{\pm}$, удовлетворяющими соотношению \eqref{es:plog} теоремы \ref{th14}, могут быть выбраны в виде представления Вейерштрасса\,--\,Адамара \cite[4.2]{HK}
 \begin{equation}\label{reprvpm}
v_{\pm}(z)=\int K_q(\zeta,z) \dd\, (\nu^{\bal[q]})^{\pm}(\zeta)+H_{\pm}(z), \quad v=v_+-v_-+H,
\end{equation}
а  при $\ord_{\infty}[v]\leq p$ функция $H\in \Har (\CC)$ --- гармонический многочлен   степени не выше $q$. Рост субгармонических функций \eqref{reprvpm} легко оценивается сверху, а именно: при нецелом $p$ получаем п.~\eqref{vvi}, а при целом $p$ --- п.~\eqref{vvii} с соотношением \eqref{llogv} подобно оценкам из \cite[лемма 4.1, теорема 4.2]{HK}. В условиях п. \eqref{vviiA} по теореме \ref{th14}, п. \ref{c:0nuiii}, меры  $(\nu^{\bal[q]})^{\pm}$ конечного типа при порядке $p$
и вновь из представления \eqref{reprvpm} стандартным путем \cite[лемма 4.1, теорема 4.2]{HK} получаем \eqref{llogvA}.
\end{proof}

\subsection{Выметание $\delta$-субгармонической функции конечного порядка из угла}\label{babdu}

Редукцией угла $\angle(\alpha , \beta)$ из \eqref{abr}  к верхней полуплоскости с помощью конформной замены переменных \eqref{tildez} с последующим обратным <<возвращением>> к углу $\angle (\alpha,\beta)$ можно распространить теорему \ref{thdsb} предыдущего подраздела \ref{bduus} с помощью теоремы 
\ref{th15} из п.~\ref{ss:ang} на выметание $\delta$-субгармонической функции из угла $\angle (\alpha, \beta)$ вида \eqref{abr}. Определение \ref{df:Bv} переносится на углы практически дословно:

\begin{definition}\label{df:Bghang}
Пусть $v\in \dsbh_*(\CC)$. Функцию $v^{{\Bal}}\in \dsbh_*(\CC)$ называем {\it выметанием функции\/ $v$ из угла\/ $\angle(\alpha, \beta)$ на его дополнение $\CC\setminus \angle(\alpha, \beta)$}, если 
$v^{{\Bal}}=v$ на $\CC\setminus \angle(\alpha, \beta)$ вне полярного множества и  $v^{{\Bal}}\bigm|_{\angle(\alpha, \beta)}$ --- гармоническая функция в угле  $\angle(\alpha, \beta)$.
\end{definition}
 
\begin{theorem}\label{thdsb_a}  
Пусть $v\in \dsbh_*(\CC)$, ее заряд Рисса $\nu:=\nu_v\in \mathcal M(\CC)$ конечного типа при порядке $p\in \RR_*^+$, а число $q$ выбрано как в \eqref{dgen}.  Для произвольного $r_0\in \RR_*^+$ представим заряд Рисса $\nu:=\nu_v$ функции $v$ в виде суммы зарядов 
\begin{equation*}
\nu=\nu\bigm|_{D(r_0)\cap \angle(\alpha,\beta)}+\nu\bigm|_{\CC \setminus  (D(r_0)\cap \angle(\alpha,\beta))}=:\nu_0+\nu_{\infty}.
\end{equation*}
Тогда существует  выметание $v^{{\Bal}}$ с зарядом Рисса
 \begin{equation*}
(\nu_0)_{\CC\setminus\angle(\alpha,\beta)}^{\bal[0]}+ (\nu_{\infty})_{\CC\setminus \angle(\alpha,\beta)}^{\bal[q]},
\end{equation*}
 представимое в виде  $v^{{\Bal}}=v_+-v_-+H$, где
$H\in \Har(\CC)$, обе функции  $v_{\pm}\in \sbh_*(\CC)$ порядка $p$ гармоничны в ${\angle(\alpha,\beta)}$, а также 
\begin{enumerate}[{\rm (i)}]
\item\label{vvii++} при любом  $p$ выполнено   \eqref{llogv};

\item\label{vvii+}   при нецелом 
$\frac{\beta-\alpha}{\pi}p$ и  целом $p\in \NN_0$ выполнено   
\eqref{llogvA}; 

\item\label{vvi+}  при одновременно нецелых $p$  и $\frac{\beta-\alpha}{\pi}p$ имеем  $\type_p^{\infty}[v_{\pm}]<+\infty$;
\item\label{vvii+++}  
 Если  $q=\frac{\beta-\alpha}{\pi}p\in \NN_0$ --- целое число, --- в частности, $\frac{\pi q}{\beta-\alpha}=p$, --- и для $\nu$ в угле $\angle(\alpha,\beta)$
 выполнено условие\/ Бляшке рода $p$, т.\,е. \eqref{ka+}, то выполнено соотношение \eqref{llogvA}.
\end{enumerate}
Если функция $v$ порядка не выше $p$, т.\,е. $\ord_{\infty}[v]\leq p$, то в качестве $H$ можно выбрать гармонический многочлен степени не выше $[p]$.
\end{theorem}

Явные выкладки с заменами переменных при выводе теоремы \ref{thdsb_a} из теоремы \ref{thdsb} громоздки, хотя идейно не вносят ничего нового по сравнению с доказательством теоремы \ref{thdsb}. Незначительное чисто техническое изменение
 после редукции \eqref{tildez} к $\CC^{\up}$ возникает для  аналога леммы \ref{lKqd}, который приходится доказывать с заменой переменных $\zeta$ и $z$ на некоторую их степень, явно выражаемую через $p,\alpha, \beta$. Здесь мы эти детали опускаем.   

\subsection{Условие Бляшке конечного рода в терминах $\delta$-субгармо\-н\-и\-ч\-е\-с\-к\-ой функции}\label{CBsf}

Ключевое для частей теорем \ref{thdsb}\eqref{vviiA}  и \ref{thdsb_a}\eqref{vvii+++} условие Бляшке рода $p$  \eqref{ka+} формулируется в терминах заряда  Рисса $\nu_v$ функции $v\in \dsbh(\CC)$. Естественно рассмотреть возможность его формулировки в терминах самой функции $v$ без привлечения её заряда Рисса $\nu_v$. 

Всюду в настоящем  подразделе \ref{CBsf} в обозначениях  \eqref{abr} предполагаем, что числа $p, q\in \RR^+$ связаны условиями
\begin{subequations}\label{pqab}
\begin{align}
q&=\frac{\beta-\alpha}{\pi}\,p \in \NN,
\tag{\ref{pqab}a}\label{pqab1}
\\
\frac{\pi}{p}&= \frac{\beta-\alpha}{q}  \,.
\tag{\ref{pqab}b}\label{pqab2}
\end{align}
\end{subequations} 

Важную роль далее будет играть интеграл 
\begin{equation}\label{fK:abp+}
J_{\alpha,\beta}^{[p]}(r_0,r;v)
:=\int_{r_0}^r\frac{v(te^{i\alpha})+(-1)^{q-1}v(te^{i\beta})}{t^{p+1}} \dd t, \quad 
0<r_0\leq r<+\infty.
\end{equation}

\begin{propos}\label{LKcr} Пусть заряд  Рисса $\nu_v$ функции $v\in \dsbh_*(\CC)$ конечного типа при порядке $p$. При  \eqref{pqab} заряд $\nu_v$ удовлетворяет условию Бляшке рода $p$ из \eqref{ka+} в угле $\angle (\alpha, \beta)$, если и только если при некотором (любом)\/ $r_0\in \RR^+_*$ выполнено соотношение
\begin{subequations}\label{fK:ab} 
	\begin{align} 
&A_{\alpha,\beta}^{[p]}(r_0,r,v)+B_{\alpha,\beta}^{[p]}(r_0,r,v)= O(1) \quad\text{при
$r\to+\infty$, где} \tag{\ref{fK:ab}a}\label{fK:aba} \\
&A_{\alpha,\beta}^{[p]}(r_0,r,v):=\frac{p}{2\pi}\int_{r_0}^r\biggl(\,\frac{1}{t^p
}
-\frac{t^p
}{r^{2p}
}\biggr)
\bigl(v(te^{i\alpha})+(-1)^{q-1}v(te^{i\beta})\bigr)
\frac{\dd t}{t} \tag{\ref{fK:ab}b}\label{fK:abb} \\
&\overset{\eqref{fK:abp+}}{=}\frac{p}{2\pi}
\biggl(J_{\alpha,\beta}^{[p]}(r_0,r;v)
-\frac{1}{r^{2p}}
\int_{r_0}^r \frac{v(te^{i\alpha})+(-1)^{q-1}v(te^{i\beta})}{t^{1-p}}
\dd t\biggr)
\tag{\ref{fK:ab}c}\label{fK:abb+}
 \\
&\overset{\eqref{fK:abp+}}{=}\frac{p^2}{\pi r^{2p}}\int_{r_0}^r
J_{\alpha,\beta}^{[p]}(r_0,t;v)\,t^{2p-1}
\dd t, \tag{\ref{fK:ab}d}\label{fK:abd}
\\
&B_{\alpha,\beta}^{[p]}(r_0,r,v):=\frac{p}{\pi
r^{p}}\int_{\alpha}^{\beta} v(re^{i\theta})\sin
p(\theta-\alpha) \dd \theta. 
\tag{\ref{fK:ab}e}\label{fK:abc} 
\end{align}
\end{subequations} 
При этом если функция $v$ представима в виде разности двух субгармонических функций конечного типа при порядке $p$, 
 то условие  Бляшке рода $p$ в угле $\angle  (\alpha ,\beta)$  для заряда Рисса $\nu_v$ эквивалентно  условию 
\begin{equation}\label{fK:abp}
J_{\alpha,\beta}^{[p]}(r_0,r;v) \overset{\eqref{fK:abp+}}{=}
O(1) \quad \text{при $r\to +\infty$}
\end{equation} 
и, более того, при всех значениях $r_0\leq r<R<+\infty$
\begin{equation}\label{fK:abp+d}
\Biggl|J_{\alpha,\beta}^{[p]}(r,R;v) 
-\int\limits_{(D(R)\setminus D(r))\cap \angle(\alpha,\beta)} \Im \frac{-1}{(ze^{-i\alpha})^p}\dd \nu_v(z)\Biggr|
=O(1).
\end{equation} 
\end{propos}
\begin{proof} 
Представление \eqref{fK:abb+} для $A_{\alpha,\beta}^{[p]}(r_0,r,v)$ сразу следует из определений 
 \eqref{fK:abb} и \eqref{fK:abp+}. 
Перейти от \eqref{fK:abb} к \eqref{fK:abd} можно аналогично \cite[(5.07)]{Levin56}
интегрированием по частям, исходя из  равенств
\begin{multline*}
A_{\alpha,\beta}^p(r_0,r,v)\overset{\eqref{fK:abb}}{=}
\frac{p}{2\pi}\int_{r_0}^r \biggl( 1
-\Bigl(\frac{t}{r}\Bigr)^{2p}\biggr) \dd
\int_{r_0}^t\frac{v(se^{i\alpha})+(-1)^{q-1}v(se^{i\beta})}{t^{p+1}} \dd s
\\
=\frac{p}{2\pi}
\,2p \int_{r_0}^r 
\biggl(\int_{r_0}^t\frac{v(se^{i\alpha})+v(se^{i\beta})}{s^{p+1}} \dd s\biggr)
\frac{t^{2p-1}}{r^{2p}} \dd t,
\end{multline*}
где внутренний интеграл  равен $J_{\alpha,\beta}(r_0,t; v)$, $t\geq r_0$.
	
Можно ограничиться рассмотрением субгармонической функции $v$ c мерой Рисса $\nu_v$ конечного типа при порядке $p$.

В условиях  \eqref{abr} и  \eqref{pqab} угол $\angle (\alpha,\beta)$ можно разбить на $q$ равных не пересекающихся  углов
\begin{equation}\label{distab}
\angle(\alpha_k,\beta_k),\quad k=1,2, \dots, q, \quad  \alpha_k=\alpha+(k-1)\,\frac{\pi}{p}, \quad \beta_k=\alpha_{k+1} 
\end{equation}  
раствора \eqref{pqab2}. Применим в каждом угле $\angle(\alpha_k,\beta_k)$ формулу Карлемана в виде 
\cite[лемма 6.1]{KhI} с конформным переносом в угол $\angle(\alpha,\beta)$. В обозначениях из \cite[предложение 6.2, (6.12)]{KhI} в случае $\type_p^{\infty}[\nu_p]<+\infty$ это дает
\begin{itemize}
\item[{[$\rm Eq_k$]}] соотношение
\begin{multline*}
A_{\alpha_k,\beta_k}(r_0,r,v)+B_{\alpha_k,\beta_k}(r_0,r,v)
\\=
 \int\limits_{(D(r)\setminus D(r_0))\cap \angle(\alpha_k,\beta_k)} \Im \biggl(\frac{-1}{(ze^{-i\alpha_k})^{\frac{\pi}{\beta_k-\alpha_k}}}
-\frac{(ze^{-i\alpha_k})^{\frac{\pi}{\beta_k-\alpha_k}}}{r^{\frac{2\pi}{\beta_k-\alpha_k}}}\biggr)\dd \nu_v(z)+O(1) \\
=\biggl|\text{ввиду
$\type_{p}^{\infty}[\nu_v]<+\infty$}\biggr|\\
=\int\limits_{(D(r)\setminus D(r_0))\cap \angle(\alpha_k,\beta_k)}
\Im\frac{-1}{(ze^{-i\alpha_k})^{\frac{\pi}{\beta_k-\alpha_k}}}\dd \nu_v(z)+O(1) \\
\overset{\eqref{distab}}{=}
\int\limits_{(D(r)\setminus D(r_0))\cap \angle(\alpha_k,\beta_k)}
\Im\frac{-1}{(ze^{-i\alpha})^{p}(-1)^{k-1}}\dd \nu_v(z)+O(1) \quad \text{при $r\to+\infty$},
\end{multline*}
\end{itemize}
Оставляя в этих $q$ равенствах ${\rm [\rm Eq_k]}$ только правые и левые части,
произведем над этими  равенствами действия вида
$\sum_{k=1}^q (-1)^{k-1} \cdot {\rm [\rm Eq_k]}$,
которые ввиду равенств
\begin{align*}
\sum_{k=1}^q (-1)^{k-1}A_{\alpha_k,\beta_k}(r_0,r,v)
&\overset{\eqref{pqab},\eqref{distab}}{=}A_{\alpha,\beta}^{[p]}(r_0,r,v),\\
\sum_{k=1}^q (-1)^{k-1}B_{\alpha_k,\beta_k}(r_0,r,v) 
&\overset{\eqref{pqab},\eqref{distab}}{=}B_{\alpha,\beta}^{[p]}(r_0,r,v),\\
\sum_{k=1}^q (-1)^{k-1}\int\limits_{(D(r)\setminus D(r_0))\cap \angle(\alpha_k,\beta_k)}& \Im\frac{-1}{(ze^{-i\alpha})^{p} (-1)^{k-1}}\dd \nu_v
\\
\overset{\eqref{ka+}}{=}\int\limits_{(D(r)\setminus D(r_0))\cap \angle(\alpha,\beta)} &\Im \frac{1}{(ze^{-i\alpha})^p}\dd \nu (z) 
\end{align*}
дают соотношение
\begin{multline}\label{eq:AB}
A_{\alpha,\beta}^{[p]}(r_0,r,v)+B_{\alpha,\beta}^{[p]}(r_0,r,v)\\=\int\limits_{(D(r)\setminus D(r_0))\cap \angle(\alpha,\beta)} \Im \frac{-1}{(ze^{-i\alpha})^p}\dd \nu_v(z)+O(1)\quad \text{при $r\to +\infty$.}
\end{multline}  
Отсюда следует эквивалентность соотношения 
\eqref{fK:aba} условию Бляшке \eqref{ka+} рода $p$ в угле $\angle(\alpha,\beta)$.

Если $\type_p^{\infty}[v]<+\infty$, то по \cite[лемма 6.2, (6.17)]{KhI} в \eqref{eq:AB} второе слагаемое в левой части можно заменить на $O(1)$ при $r\to +\infty$ и в этом случае
\begin{multline}\label{JintO}
J_{\alpha,\beta}^{[p]}(r_0,r;v)
-\frac{1}{r^{2p}}
\int_{r_0}^r \bigl(v(te^{i\alpha})+(-1)^{q-1}v(te^{i\beta})\bigr)
t^{p-1} \dd t\\ 
\overset{\eqref{fK:abb+}}{=}
\int\limits_{(D(r)\setminus D(r_0))\cap \angle(\alpha,\beta)} \Im \frac{-1}{(ze^{-i\alpha})^p}\dd \nu_v(z)+O(1) \quad\text{при $r\to +\infty$}
\end{multline}
При этом из \eqref{abr} и \eqref{pqab1} следует, что $p\geq 1/2$.
\begin{lemma}\label{eray}
Пусть $p\in \RR^+$ и $v\in \sbh_*(\CC)$, $\type_p^{\infty}[v]<+\infty$. Тогда для любого $r_0\in \RR^+_*$ справедливо соотношение
\begin{equation}\label{rps}
\int_{r_0}^{r} \bigl|v(t)\bigr|t^{s-1} \dd t=
\begin{cases}
O\bigl(r^{p+s^+}\bigr)&\text{ для $s\in \RR_*$},\\
O\bigl(r^{p}\log r\bigr)&\text{ для $s=0$}
\end{cases}
\quad \text{при $r\to +\infty$}. 
\end{equation}
\end{lemma}
\begin{proof}[Доказательство леммы \ref{eray}]
При $s\geq 1$ достаточно применить теорему  A3  с оценкой \eqref{ginc}, а при $s<1$ --- 
ту же теорему  A3 c  оценкой  \eqref{ubg}.
\end{proof}

По лемме \ref{eray} с $s=p$ вычитаемый интеграл в левой части \eqref{JintO} c множителем $r^{-2p}$ ведет себя как $O(1)$ при $r\to +\infty$. Таким образом, в рассматриваемом случае выполнено \eqref{fK:abp+d}, согласно которому  условие  Бляшке рода $p$ в угле $\angle  (\alpha ,\beta)$  для заряда Рисса $\nu_v$ эквивалентно  условию \eqref{fK:abp}. 
\end{proof} 

По аналогии с целыми функциями класса A \cite[гл.~V]{Levin56}  примем
\begin{definition}\label{df:clA} 	 Функция $v\in \dsbh_*(\CC)$ c зарядом Рисса $\nu_v$ конечного типа при порядке $p\in \RR^+$ принадлежит {\it классу Ахиезера  рода $p$,\/} или {\it классу\/ {\rm A} рода $p$, относительно угла $\angle (\alpha, \beta)$} в условиях \eqref{pqab1}, если выполнено одно  из двух эквивалентных друг другу в силу предложения \ref{LKcr}    условий:
\begin{enumerate}[{\rm (i)}]
\item\label{Aci} заряд  Рисса $\nu_v$ функции $v$ удовлетворяет условию Бляшке рода $p$  из \eqref{ka+} в угле $\angle (\alpha, \beta)$; 
\item\label{Acii} функция $v$ удовлетворяет условию \eqref{fK:aba} с соответствующими упрощением \eqref{fK:abp}
при $\type_p^{\infty}[v]<+\infty$.
\end{enumerate}
\end{definition}

\begin{remark}\label{rgen} При доказательстве предложения \ref{LKcr} можно было бы использовать и некоторое обобщение формулы Карлемана без разбиения угла $\angle(\alpha, \beta)$ вида \eqref{distab}, основанное при условиях \eqref{pqab} на дважды примененной второй формуле Грина к функции $v\in \sbh_*(\CC)$ и к функции
\begin{equation*}
\Im \biggl(\frac{1}{(ze^{-i\alpha})^{p}}
-\frac{(ze^{-i\alpha})^{p}}{r^{2p}}\biggr),
\end{equation*} 
а также ее инверсии относительно $\partial D(r_0)$, $r_0<r$, так, как это проделано при доказательстве \cite[теорема 1]{Kha88}.  
\end{remark}

\begin{remark}\label{remRay+}
Для функции $v\in \dsbh_*(\CC)$,   представимой в виде разности субгармонических функций конечного типа при порядке $p\in \RR_*^+$, c зарядом Рисса $\nu=\nu_v$, исходя из классического выметания из первой части \cite{KhI} и конечного числа выметаний для каждого дополнительного к системе лучей угла раствора не менее $\pi/p$, можно построить {\it глобальное выметание
$v_{S}^{{\Bal}}$ на систему лучей $S$}, 
по простой схеме:{\it
\begin{enumerate}[{\rm (i)}]
\item\label{bSiii+} произвести классическое выметание $v^{\bal}$ рода $0$ одновременно из всех дополнительных к $S$ углов раствора $<\pi/p$ и из конечного числа дополнительных к $S$ углов \eqref{abr}  раствора $\geq \pi/p$, в которых выполнено классическое условие Ахиезера рода $1$ около $\infty$ в угле $\angle(\alpha,\beta)$ \cite[теорема 8, предложение 6.2, определение 6.4]{KhI};  
\item\label{bSiv+} В, возможно,  оставшемся конечном числе дополнительных углов \eqref{abr} раствора $\geq \pi/p$, в которых не выполнено условие Ахиезера  рода $1$ из определения\/ {\rm \ref{df:clA},} построить выметание $(v^{\bal})^{\Bal}$ рода  $q$ так же, как в теореме\/ {\rm \ref{thdsb_a}}.
\end{enumerate}}
Последовательность применения пунктов \eqref{bSiii+} и \eqref{bSiv+} здесь  можно поменять местами.
При этом в результате действий по п.~\eqref{bSiii+} выметание $v^{\bal}$ будет оставаться функцией конечного типа при порядке $p$, к тому же субгармонической, если субгармонична исходная функция $v$. При необходимости применения п.~\eqref{bSiv+}
результатом будет $\delta$-субгармоническая функция на $\CC$ и   возможны вариации ее роста в соответствии с  пунктами \eqref{vvii++}--\eqref{vvii+++}  теоремы \ref{thdsb_a}.
Более детально, по теореме \ref{thdsb_a} имеем $v_{S}^{{\Bal}}=v_+-v_-+H$, где функции $v_{\pm}\in \sbh_*(\CC)$ удовлетворяют, как минимум, одному из ограничений   \eqref{llogv} или \eqref{llogvA}, $H\in \Har(\CC)$ многочлен степени не выше $[p]$, а  заряд Рисса выметания 
$v_{S}^{{\Bal}}$ --- это выметание $\nu_{S}^{\Bal}$ заряда  $\nu$, построенное по соответствующей схеме \eqref{bSi}--\eqref{bSiv} из замечания \ref{remRay}.
\end{remark}

\subsection{Доказательство теоремы A2}\label{A2} 
Части \eqref{vvii++R}, \eqref{vvi+R} и 
\eqref{vvii+++R} теоремы A2 --- частные случаи соответственно частей \eqref{vvii++},  \eqref{vvi+} и \eqref{vvii+++} теоремы \ref{thdsb_a} при $\beta-\alpha =2\pi$, где дополнение 
\eqref{fK:abp+0}  основано на условии \eqref{fK:abp} предложения \ref{LKcr}.    
В случае  $p\in \NN$ можно дать два варианта  вывода части \eqref{vvii+R} теоремы A2. 
В обоих вариантах ввиду $q=2p\in\NN$ основываемся на части \eqref{vvii+++}  теоремы \ref{thdsb_a}, согласно которой для доказательства   соотношения \eqref{llogv++} достаточно убедиться, что 
для заряда  Рисса $\nu_v$ функции $v\in \dsbh_*(\CC)$, представимой в виде разности двух {\it субгармонических функций конечного типа при порядке\/} $p$, выполнено 
условие Бляшке \eqref{ka+}, т.\,е. \eqref{ka+0}, рода $p$ в угле $\angle(0,2\pi)$. Из такого представления функции $v$ сразу следует, что для  $\nu_v$ выполнено условие Линделёфа \cite[определение 4.7]{KhI}
\begin{equation*} 
\Biggl|\;\int\limits_{D(r)\setminus D(r_0)} \frac{1}{z^p}
\dd \nu(z) \Biggr|=O(1) \quad\text{при $r\to +\infty$,} 
\end{equation*}
из которого ввиду равенства 
\begin{equation*}
\Biggl|\;\int\limits_{(D(r)\setminus D(r_0))\setminus \RR^+} \Im \frac{1}{z^{p}}\dd \nu (z)\Biggr|
=\Biggl|\; \Im \int\limits_{D(r)\setminus D(r_0)} \frac{1}{z^p}
\dd \nu(z) \Biggr|, \quad p\in \NN,
\end{equation*}
при {\it целом} $p\in \NN$ сразу следует условие Бляшке \eqref{ka+0} рода $p$ в угле $\angle(0,2\pi)$.

 Другой подход опирается на то, что для угла $\angle (0,2\pi)$ при $p\in \NN$ имеем четное $q=2p$  и  тогда $J_{0,2\pi}^{[p]}(r_0,r;v)\overset{\eqref{fK:abp+}}{\equiv} 0$ при $r\geq r_0$. Следовательно, автоматически выполнено условие  \eqref{fK:abp}, эквивалентное  условию 
Бляшке рода $p$ в угле $\angle  (0,2\pi)$  для заряда Рисса $\nu_v$, т.\,е. снова получаем требуемое условие Бляшке \eqref{ka+0} рода $p$ в угле $\angle (0, 2\pi )$.

\subsection{Интегралы от субгармонических функций по лучу}\label{Intr}
Результаты настоящего подраздела представляют и самостоятельный интерес, поскольку необходимость оценки роста интегралов по лучу от субгармонических функций  довольно часто возникает в теории функций. Поэтому мы устанавливаем их в гораздо более общей форме, чем требуется для наших целей в нашей работе: для субгармонических функций произвольного роста с произвольными положительными функциями-множителями. Техника оценок имеет параллели с доказательством леммы Эдрея и Фукса о малых дугах \cite{EF}, \cite[гл. 1, теорема 7.3]{GO}, леммы о малых интервалах А.\,Ф.~Гришина и М.\,Л.~Содина \cite[лемма 3.2]{GrS}   и ее аналогов \cite[теорема 8]{GrM}.   
  
Нам потребуется одна из оценок снизу для субгармонических функций, которые установлены в \cite{Kh84d}, \cite{Kh84} в гораздо более общей форме, чем необходимо в настоящем подразделе. Приведем здесь эту оценку снизу в существенно упрощенной форме, используя обозначение $D(z,r):=z+D(r)$ для открытого круга с центром в точке $z\in\CC$ радиуса $r\in \RR_*^+$.

\begin{propos}[{\rm \cite[теорема 1.2]{Kh84d}, \cite[теорема 2]{Kh84}}]\label{prsn}
Пусть $v\in \sbh_*(\CC)$, число $B>1$. Тогда существуют абсолютная постоянная $A\geq 1$
и  последовательность кругов $D(z_j,r_j)$, $j\in \NN$, для которых
\begin{subequations}\label{esl}
\begin{align}
v(z)&\overset{\eqref{Mvr}}{\geq} -M_v\bigl((1+1/B)|z|\bigr)AB\log (AB), \quad z\notin \bigcup_{j\in\NN}D(z_j,r_j), 
\tag{\ref{esl}a}\label{esla}
\\
 \sum_{|z_j|\leq r}r_j&\leq \frac{1}{B}\, r\quad \text{при всех $r>0$}. 
\tag{\ref{esl}b}\label{eslb}
\end{align}
\end{subequations}
\end{propos}
Предложение \ref{prsn} сразу следует из   \cite[теорема 1.2]{Kh84d}, \cite[теорема 2]{Kh84} при  постоянных функциях $A(t)\equiv B^2$, $\psi (t)\equiv 1/B$, $t\in \RR^+$, и числе $\beta=1$. 

\begin{propos}\label{disk}
Для любых двух чисел  $r_0, b\in \RR^+$ из \eqref{r0ab}
 можно подобрать  столь большое число $B>1$, что для любой последовательности кругов  $D(z_j,r_j)$, $j\in \NN$, удовлетворяющей  ограничению \eqref{eslb}, для любой точки $z\notin D(r_0)$ найдется число
$r_z>0$, удовлетворяющее условиям 
\begin{equation}\label{rqabD}
0< r_z \leq b|z|,
\quad  \biggl(\; \bigcup_{j\in\NN}D(z_j,r_j)\biggr)
\bigcap \partial D(z,r_z)
=\varnothing,
\end{equation}
т.\,е. окружность $\partial D(z,r_z)$ при некотором $r_z>0$ содержится в  круге $D\bigl(z, b|z|\bigr)$ и не пересекается ни с одним из кругов $D(z_j,r_j)$, $j\in \NN$.
\end{propos}Стандартное элементарное доказательство предложения  \ref{disk}, основанное на круговых проекциях  $D(z_j,r_j)$ на радиус круга $D\bigl(z,b|z|\bigr)$, опускаем.

\begin{propos}\label{lemrep}
Пусть $v\in \sbh_*(\CC)$ c мерой Рисса $\nu_v$. Тогда в  условиях \eqref{r0ab} найдется число $C\geq 1$, с которым справедлива оценка 
\begin{equation*}
\bigl|v(z)\bigr|\leq \int\limits_{(1-b)|z|}^{(1+b)|z|}
\log^+\frac{b |z|}{\bigl|r-|z|\bigr|} \dd \nu_v^{\rad}(r)+CM_v^+\bigl((1+2b)|z|\bigr)\quad \text{при всех $|z|>r_0$,}
\end{equation*}
где $\log^+r:=\max\{0,\log r \} $.
\end{propos}
\begin{proof} По определению \eqref{Mvr} из принципа максимума 
\begin{equation}\label{est_ab}
v(z)\leq M_v\bigl(|z|\bigr)\leq M_v^+\bigl((1+2b)|z|\bigr), \quad z\in \CC.
\end{equation} 
Исходя из оценки снизу \eqref{esla} предложения \ref{prsn} 
при достаточно большом $B$ по предложению \ref{disk}
найдется число $r_z\in \bigl(a|z|, b|z|\bigr)$, удовлетворяющее  согласно \eqref{rqabD} для некоторой постоянной $C\in \RR$, не зависящей от точки $z\notin D(r_0)$, оценке снизу 
\begin{multline}\label{es_rzM}
v(z')\geq -C M_v\bigl((1+1/B)|z'|\bigr)\geq -C M_v^+\bigl((1+1/B)(1+b)|z|\bigr)\\
\geq -C M_v^+\bigl((1+2b)|z|\bigr)
\quad \text{для всех $z'\in \partial D(z,r_z)$}.
\end{multline}
По формуле Пуассона\,--\,Йенсена \cite[3.7]{HK}, \cite[4.5]{Rans} для круга $D(z,r_z)$ имеем
\begin{equation}\label{fPJv}
v(z)=-\int_{D(z,r_z)} \log\Bigl|\frac{r_z^2}{r_z(z-z')}\Bigr|\dd \nu_v(z')+h(z)=:G(z)+h(z),
\end{equation}
где наименьшая гармоническая мажоранта $h$ функции $v$ в круге $D(z,r_z)$ ввиду \eqref{est_ab} и \eqref{es_rzM} удовлетворяет оценке 
\begin{equation}\label{esh}
\bigl|h(z)\bigr|\leq C M_v^+\bigl((1+2b)|z|\bigr),
\end{equation}
и, поскольку функция Грина под знаком интеграла c полюсом $z$ в центре круга в промежуточном равенстве  из \eqref{fPJv} положительна, имеем 
\begin{multline*}
\bigl|G(z)\bigr|\leq \int_{D(z,r_z)}
\log\frac{r_z}{|z-z'|}\dd \nu_v(z')\leq \int_{D(z,r_z)}\log^+\frac{b|z|}{|z-z'|}
\dd \nu_v(z') 
\\
\leq \int_{D(z,b|z|)}\log^+\frac{b|z|}{\bigl||z'|-|z|\bigr|}
\dd \nu_v(z')\leq \int\limits_{(1-b)|z|}^{(1+b)|z|}
\log^+\frac{b|z|}{\bigl|r-|z|\bigr|} \dd \nu_v^{\rad}(r).
\end{multline*} 
Из последней оценки, \eqref{esh} и \eqref{fPJv} следует    
 предложение \ref{lemrep}.
\end{proof}

\begin{theorem}\label{thr}
Пусть $v\in \sbh_*(\CC)$ c мерой Рисса $\nu_v$, $g\colon \RR_*^+\to \RR^+$ --- локально интегрируемая функция по мере Лебега $\lambda_{\RR}$ на $\RR^+$ и числа $r_0,b$ из  \eqref{r0ab}. Тогда найдется постоянная $C\geq 1$, с которой выполнена оценка
\begin{multline}\label{Intest}
\int\limits_{r_0}^R|v(t)|\, g(t) \dd t \leq 2R \int\limits_{(1-b)r_0}^{(1+b)R}\sup\bigl\{g(s)\colon (1-2b)r\leq s\leq (1+2b)r\bigr\}
 \dd \nu_v^{\rad}(r)\\+CM_v^+\bigl((1+2b)R\bigr)
\int\limits_{r_0}^R g(t) \dd t\quad \text{для всех $R\geq r_0$}.
\end{multline}
\end{theorem}
\begin{proof} Из оценки предложения \ref{lemrep} 
\begin{multline}\label{es1}
\int\limits_{r_0}^R|v(t)| g(t) \dd t \leq
\int\limits_{r_0}^R\Biggl(\; \int\limits_{(1-b)t}^{(1+b)t}
\log^+\frac{b t}{|r-t|} \dd \nu_v^{\rad}(r)+CM_v^+\bigl((1+2b)t\bigr)\Biggl)g(t)\dd t
\\
\leq \int\limits_{r_0}^R\Biggl(\; \int\limits_{(1-b)t}^{(1+b)t}
\log^+\frac{R}{|r-t|} \dd \nu_v^{\rad}(r)\Biggr) g(t)\dd t
+CM_v^+\bigl((1+2b)R\bigr)\int\limits_{r_0}^R g(t)\dd t\\
=I(r_0,R)+CM_v^+\bigl((1+2b)R\bigr)\int\limits_{r_0}^R g(t)\dd t
\quad\text{при всех $R\geq r_0$}, 
\end{multline}
где через $I(r_0,R)$ обозначен повторный интеграл перед равенством. Для его оценки поменяем интегралы местами:
\begin{multline}\label{esIrR}
I(r_0,R)\leq\int\limits_{(1-b)r_0}^{(1+b)R}\int_{\frac{r}{1+b}}^{\frac{r}{1-b}} \log^+\frac{R}{|t-r|}\,g(t)\dd t \dd \nu_v^{\rad}(r)\\
\leq\int\limits_{(1-b)r_0}^{(1+b)R}\int_{(1-2b)r}^{(1+2b)r} \log^+\frac{R}{|t-r|}\,g(t)\dd t \dd \nu_v^{\rad}(r).
\end{multline}
Используя замену $t-r=s$ в последнем внутреннем интеграле, можем продолжить эти неравенства как 
\begin{multline}\label{esIrR+}
\leq \int\limits_{(1-b)r_0}^{(1+b)R}\sup\bigl\{g(s)\colon (1-2b)r\leq s\leq (1+2b)r\bigr\}
\Biggl(\,2\int_0^{2br}\log^+\frac{R}{s} \dd s\Biggr) \dd \nu_v^{\rad}(r)\\
\leq \int\limits_{(1-b)r_0}^{(1+b)R}\sup\bigl\{g(s)\colon (1-2b)r\leq s\leq (1+2b)r\bigr\}
\Biggl(\,2\int_0^{R}\log^+\frac{R}{s} \dd s\Biggr) \dd \nu_v^{\rad}(r)
\\
=\int\limits_{(1-b)r_0}^{(1+b)R}\sup\bigl\{g(s)\colon (1-2b)r\leq s\leq (1+2b)r\bigr\}
2R \dd \nu_v^{\rad}(r).
\end{multline}
Таким образом, \eqref{es1},\eqref{esIrR}, \eqref{esIrR+} дают требуемую оценку \eqref{Intest}.
\end{proof}

\subsection{Доказательство теоремы A3}\label{prii}
Для {\it возрастающей\/}  $g$ из 
оценки \eqref{Intest} теоремы \ref{thr} следует цепочка неравенств 
 \begin{multline*}
\int\limits_{r_0}^R|v(t)|\, g(t) \dd t \leq 2R \int\limits_{(1-b)r_0}^{(1+b)R} g\bigl((1+2b)r\bigr)
 \dd \nu_v^{\rad}(r)+CM_v^+\bigl((1+2b)R\bigr) g(R)R
\\
\leq 2R\, g\bigl((1+4b)R\bigr)\nu_v^{\rad}\bigl((1+b)R\bigr)
+CM_v^+\bigl((1+2b)R\bigr) g(R)R
\\
\leq 2R\, g\bigl((1+4b)R\bigr)M_v^+\bigl((1+2b)R\bigr)\log\frac{1+2b}{1+b}\\
+CM_v^+\bigl((1+2b)R\bigr) g(R)R\quad \text{для всех $R\geq r_0$,}
\end{multline*}
 что доказывает \eqref{ginc}.

Для {\it убывающей\/} функции  $g$ из оценки \eqref{Intest} теоремы \ref{thr}, используя интегрирование по частям, получаем
\begin{multline}\label{ivtl}
\hspace{-1mm}\int\limits_{r_0}^R|v(t)|\, g(t) \dd t \leq 2R \int\limits_{(1-b)r_0}^{(1+b)R} g\bigl((1-2b)r\bigr)\dd \nu_v^{\rad}(r)+CM_v^+\bigl((1+2b)R\bigr)
\int\limits_{r_0}^R g(t) \dd t\\
\leq 2R\, g\bigl((1-2b)(1+b)R\bigr) \nu_v^{\rad}\bigl((1+b)R\bigr)+2R\int\limits_{(1-b)r_0}^{(1+b)R} 
\nu_v^{\rad}(r) \dd \, \Bigl(-g\bigl((1-2b)r\bigr)\Bigr)
\\+CM_v^+\bigl((1+2b)R\bigr) \int\limits_{r_0}^R g(t) \dd t
\leq 
2R\, g\bigl((1-b)R\bigr) M_v^+\bigl((1+2b)R\bigr)
\log\frac{1+2b}{1+b}\\
+2R M_v^+\bigl((1+2b)R\bigr) \log\frac{1+2b}{1+b}
\Bigl(-g\bigl((1-2b)(1+b)R\bigr)+g\bigl((1-2b)(1-b)r_0\bigr)\Bigr)\\
 +CM_v^+\bigl((1+2b)R\bigr) \int\limits_{r_0}^R g(t) \dd t
\quad \text{для всех $R\geq r_0$}.
\end{multline}
Ввиду строгой {\it положительности\/} функции $g$, очевидно,
\begin{equation}\label{iincr+}
\int_{r_0}^{R}g(t) \dd t \quad\text{--- {\it возрастающая\/} функция  по $R\geq r_0$}.
\end{equation}
Цепочку неравенств \eqref{ivtl} можем продолжить  как 
\begin{multline*}
\leq CM_v^+\bigl((1+2b)R\bigr)
\Bigl(C_bR\, g\bigl((1-b)R\bigr)+C_b'g\bigl((1-3b)r_0\bigr) 
+C\int\limits_{r_0}^R g(t) \dd t \Bigr)
\\
\leq CM_v^+\bigl((1+2b)R\bigr)
\Bigl(C_b\int_{r_0}^R g\bigl((1-b)r\bigr) \dd r
+C_br_0g\bigl((1-b)R\bigr)+C_b'g\bigl((1-3b)r_0\bigr) 
\Bigr.\\
\Bigl.+C\int\limits_{r_0}^R g(t) \dd t\Bigr)
\overset{\eqref{iincr+}}{\leq}
C_{b,r_0}M_v^+\bigl((1+2b)R\bigr)
\int_{r_0}^R g\bigl((1-b)r\bigr) \dd r
\quad\text{при $R\geq 2r_0$},  
\end{multline*}
где $C_b,C_b', C_{b,r_0}$ --- постоянные, не зависящие от $R\geq 2r_0$. Это дает \eqref{ubg}.

\subsection{Критерии вполне регулярного роста субгармонической  фу\-н\-к\-ц\-ии на системе лучей}\label{crcrS}

Известно, что субгармоническая функция конечного типа при порядке $p\in \RR_*^+$ имеет вполне регулярный рост на замкнутой системе лучей $S$ тогда и только тогда, когда она вполне регулярного роста на каждом из лучей из $S$ \cite[введение]{Aza69}. Поэтому один из критериев вполне регулярного роста на $S$ содержится в следующем пункте  

\subsubsection{Критерий вполне регулярного роста субгармонической  функции на одном луче}\label{crcrg+}

Для заряда $\nu\in \mathcal M(\RR^+)$ функцию распределения этого заряда на $\RR^+$ обозначаем через
$\nu^{\RR^+}(x)=\nu \bigl([0,x]\bigr)$, $x\in \RR^+$.

\begin{theorem}\label{thcrgr}
Функция  $v\in \sbh_*(\CC)$ с мерой Рисса $\nu_v$, $0\notin \supp \nu_v$, конечного типа при порядке $p\in \RR_*^+$ имеет вполне регулярный рост при том же порядке $p\in \RR_*^+$ на луче $\RR^+$ тогда и только тогда, когда для   выметания $(\nu_v)_{\RR^+}^{\bal [[2p]]}\in \mathcal M(\RR^+)$ рода $[2p]$ меры Рисса $\nu_v$ на $\RR^+$ и для некоторого исключительного подмножества $E\overset{\eqref{cvpreg+}}{\subset} \RR^+$ нулевой относительной линейной меры на $\RR^+$ существует предел 
\begin{equation}\label{cvpreg++}
\lim_{\substack{x\to +\infty\\ x\notin E}}
x^{[p]+1-p}\fint\limits_0^{+\infty} \frac{1}{x-t}
\, \Bigl((\nu_v)_{\RR^+}^{\bal [[2p]]}\Bigr)^{\RR^+} (t) \,\frac{\dd t}{ \, t^{[p]+1}} \in \RR.
\end{equation}
\end{theorem}
\begin{proof} Выметание  $v^{\Bal}$ на $\RR^+$ субгармонической функции $v$  может быть с зарядом Рисса 
$(\nu_v)_{\RR^+}^{\bal [[2p]]}\in \mathcal M(\RR^+)$, который удовлетворяет, как минимум, асимптотическому условию из п.~\eqref{vvii++R} теоремы A2 , и выметание  $v^{\Bal}$ совпадает с функцией $v$ почти всюду по мере Лебега на $\RR^+$. При этом справедливо представление Вейерштрасса\,--\,Адамара
\begin{equation}\label{iKp+}
v(x)=v^{\Bal}(x)=\int\limits_0^{+\infty}
K_{[p]} (x,t)\dd \, \Bigl((\nu_v)_{\RR^+}^{\bal [[2p]]}\Bigr)^{\RR^+} (t) +H(x) 
\end{equation}
для почти всех $x\in \RR^+$, где $H$ --- гармонический многочлен степени не выше $[p]$, уже являющийся функцией вполне регулярного роста на $\CC$.
Согласно \cite[предложение 6.3]{KhI} интеграл в правой части \eqref{iKp+} можно записать как 
\begin{equation}\label{rep:nKq+}
\fint\limits_0^{+\infty} 
\Re \frac{x^{[p]+1}}{t^{[p]+1}(x-t)} \, \Bigl((\nu_v)_{\RR^+}^{\bal [[2p]]}\Bigr)^{\RR^+} (t) \dd t.
\end{equation}
Определение вполне регулярного роста при порядке $p$ на луче $\RR^+$ функции $v$ означает существование конечного предела 
\begin{equation}\label{vl}
\lim_{\substack{x\to +\infty\\ x\notin E}} \frac{1}{x^p}v(x),
\end{equation}
где исключительное подмножество $E\overset{\eqref{cvpreg+}}{\subset} \RR^+$ нулевой относительной линейной  меры на $\RR^+$ <<не чувствует>> добавления множества нулевой меры Лебега на $\RR^+$. Согласно \eqref{iKp+}--\eqref{rep:nKq+}  существование предела \eqref{vl} эквивалентно существованию предела \eqref{cvpreg++}.
\end{proof}

\begin{remark}
Теорема A4 из п.~\ref{crcrg} сразу следует из теоремы \ref{thcrgr} при $v=\log |f|$ и $\nu_v=n_{\sf Z}$.
\end{remark}

\subsubsection{Критерий вполне регулярного роста субгармонической  функции на конечной системе лучей}\label{crcfS+}
Для произвольной конечной системы лучей 
\begin{equation}\label{lk1_k}
\begin{split}
l_j:=l(\theta_j):=\{&te^{i\theta_j}\in \CC\colon t\in \RR^+\},  
\\ 
j=1, \dots, k\in \NN, \quad &\theta_1<\dots <\theta_k<\theta_1+2\pi,	
\end{split}
\end{equation}
для субгармонической функции $v\in \sbh_*(\CC)$ конечного типа при порядке $p\in \RR_*^+$ c мерой Рисса $\nu$ также конечного типа при порядке $p$, комбинируя выметания различного рода выметания из дополнительных к системе $S=\{l_j\colon j=1,\dots , k\in \NN\}$, как в замечаниях \ref{remRay} и \ref{remRay+},  можно явно построить выметание $v^{\Bal}_S\in \dsbh_*(\CC)$ функции $v$ на $S$ c зарядом Рисса $\nu^{\Bal}_S$
c носителем на $S$ и удовлетворяющим, как минимум, условию 
$|\nu^{\Bal}|^{\rad}(r)=O(r^p\log r)$ при $r\to +\infty$.  
Обозначим функции распределения на лучах $l_j$ из \eqref{lk1_k} выметания $\nu^{\Bal}$ через 
\begin{equation}\label{nball}
n_j(t):=\nu^{\Bal}\bigl(l_j\cap \overline{D}(t)\bigr).
\end{equation}

В приведенных соглашениях и обозначениях п.~\ref{crcfS+} имеет место
\begin{theorem}\label{thcrgfS}
Пусть $S\overset{\eqref{lk1_k}}{=}\{l_j\}_{j=1,\dots,k}$. Функция $v$ вполне регулярного роста на $S$ при порядке $p$, если  и только если существуют множества $E_j\overset{\eqref{cvpreg+}}{\subset} \RR^+$, $j=1,\dots,k$,  нулевой относительной линейной меры, для которых с функциями распределения $n_j$ 
из \eqref{nball} существуют $k$ пределов
\begin{equation}\label{c:vprreg}
\lim_{\substack{r\to +\infty\\r\notin E_j}} {r^{[p]+1-p}}\sum_{j'=1}^{k} \fint_0^{+\infty}\Re \frac{e^{i([p]+1)\theta_{j}}}{t^{[p]+1}(re^{i\theta_j}-t)} \, n_{j'}(t) \dd t, \quad j=1,\dots,k.
\end{equation}
\end{theorem}
Доказательство опускаем, поскольку оно повторяет таковое из \cite[6.4.2, доказательство теоремы 10]{KhI}.


\end{document}